\documentclass{amsart}
\usepackage[top=2.5cm,left=2cm,right=2cm,bottom=3cm]{geometry}
\usepackage[toc,page]{appendix}
\usepackage{amsmath, amssymb, amstext,amsthm}
\usepackage{enumitem}
\usepackage{stackengine} 
\usepackage{bm} 
\usepackage{hyperref}
\usepackage[british]{babel}
\usepackage{verbatim}
\usepackage{esint}
\usepackage[toc]{appendix}
\usepackage{amsfonts}\usepackage{amscd}\usepackage{url}\usepackage{amsopn}
\usepackage{bbm}
\usepackage{cite}\allowdisplaybreaks[1]
\usepackage{stmaryrd}

\makeatletter

\newcounter{thm}
\numberwithin{thm}{section}
\newtheorem{remark}[thm]{Remark}
\newtheorem{definition}[thm]{Definition}
\newtheorem{theorem}[thm]{Theorem}
\newtheorem{lemma}[thm]{Lemma}
\newtheorem{proposition}[thm]{Proposition}

\newcommand{\vt}{\vartheta}
\newcommand{\ve}{\varepsilon}
\newcommand{\vp}{\varphi}

\newcommand{\E}{\mathbb{E}}
\renewcommand{\P}{\mathbb{P}}
\renewcommand{\O}{\Omega}
\renewcommand{\o}{\omega}

\newcommand{\mcA}{\mathcal{A}}
\newcommand{\mcB}{\mathcal{B}}

\newcommand{\mcE}{\mathcal{E}}
\newcommand{\mcF}{\mathcal{F}}

\newcommand{\mcH}{\mathcal H}

\newcommand{\mcL}{\mathcal{L}}

\newcommand{\mcO}{\mathcal{O}}
\newcommand{\mcP}{\mathcal{P}}

\newcommand{\mcS}{\mathcal{S}}

\newcommand{\mcU}{\mathcal{U}}
\newcommand{\mcV}{\mathcal{V}}
\newcommand{\mcW}{\mathcal{W}}
\newcommand{\mcX}{\mathcal{X}}

\newcommand{\mfB}{\mathfrak{B}}

\newcommand{\mfP}{\mathfrak{P}}
\newcommand{\mfW}{\mathfrak{W}}

\newcommand{\mbD}{\mathbb{D}}

\newcommand{\mbH}{\mathbb{H}}
\newcommand{\mbL}{\mathbb{L}}

\newcommand{\mbT}{\mathbb{T}}
\newcommand{\mbW}{\mathbb{W}}
\newcommand{\R}{\mathbb{R}}

\newcommand{\N}{\mathbb{N}}

\newcommand{\D}{\Delta}

\newcommand{\restr}[1]{|_{#1}} 

\newcommand{\diff} {\mathrm{d}}

\DeclareMathOperator*{\supp}{supp}
\DeclareMathOperator*{\spann}{span}
\DeclareMathOperator*{\divv}{div}

\begin{document}

\title{The Stochastic Tamed MHD Equations -- Existence, Uniqueness and Invariant Measures}
\author{Andre Schenke}
\email{aschenke@math.uni-bielefeld.de}
\address{Fakult\"at f\"ur Mathematik, Universit\"at Bielefeld, 33615 Bielefeld, Germany}

\keywords{Tamed MHD equations, magnetohydrodynamics, MHD equations, porous media, invariant measure, strong solutions, Feller semigroup}
\subjclass{76W05, 76S05; 35K91, 76D03, 60H15}
\begin{abstract}
	We study the tamed magnetohydrodynamics equations, introduced recently in a paper by the author, perturbed by multiplicative Wiener noise of transport type on the whole space $\R^{3}$ and on the torus $\mbT^{3}$. In a first step, existence of a unique strong solution are established by constructing a weak solution, proving that pathwise uniqueness holds and using the Yamada-Watanabe theorem. We then study the associated Markov semigroup and prove that it has the Feller property. Finally, existence of an invariant measure of the equation is shown for the case of the torus.
\end{abstract}

\maketitle

\section{Introduction}
In this paper, we consider a randomly perturbed version of the tamed MHD (TMHD) equations introduced recently in \cite{Schenke20b}. This aims at modelling the turbulent behaviour of a flow of electrically conducting fluids through porous media. To be precise, we study existence and uniqueness of strong solutions, as well as existence of invariant measures of the following system of equations:
\begin{equation}\label{STMHD_eq_STMHD}
	\begin{cases}
		\diff \bm{v} &= \left[ \Delta \bm{v} - \left( \bm{v} \cdot \nabla \right) \bm{v}  +  \left( \bm{B} \cdot \nabla \right) \bm{B} + \nabla \left( p + \frac{ | \bm{B} |^2}{2} \right) - g_{N}(| (\bm{v}, \bm{B}) |^{2}) \bm{v} \right] \diff t\\
		&\quad + \sum_{k=1}^{\infty} \left[ (\bm{\sigma}_{k}(t) \cdot \nabla) \bm{v} + \nabla p_{k}(t) + \bm{h}_{k}(t,y(t)) \right] \diff W_{t}^{k} + \bm{f}_{v}(t,y(t)) \diff t, \\		
		\diff \bm{B}
		&= \left[ \Delta \bm{B} - \left( \bm{v} \cdot \nabla \right) \bm{B} + (\bm{B} \cdot \nabla) \bm{v} + \nabla \pi - g_{N}(| (\bm{v}, \bm{B}) |^{2}) \bm{B} \right] \diff t \\
		&\quad + \sum_{k=1}^{\infty} \left[ (\bar{\bm{\sigma}}_{k}(t) \cdot \nabla) \bm{B} + \nabla \pi_{k}(t) + \bar{\bm{h}}_{k}(t,y(t)) \right] \diff \bar{W}_{t}^{k} + \bm{f}_{B}(t,y(t)) \diff t, \\
		\divv (\bm{v}) &= 0, \quad \divv (\bm{B}) = 0.
	\end{cases}
\end{equation}
Here, $\bm{v} = \bm{v}(t,x)$ denotes the velocity field of the fluid, $\bm{B} = \bm{B}(t,x)$ is the magnetic field, $p = p(t,x)$ is the pressure, $\pi = \pi(t,x)$ is the ``magnetic pressure" as introduced in \cite[Section 1.2.3]{Schenke20b}, $g_{N}$ denotes the taming function, which can be understood as an additional drag or friction term due to the porous medium. The taming function $g_{N} \colon \mathbb{R}_{+} \rightarrow \mathbb{R}_{+}$ is defined by
\begin{equation}\label{DTMHD_eq_def_g_N}
	\begin{cases}
		g_{N}(r) := 0, \quad & r \in [0,N], \\
		g_{N}(r) := C_{\text{taming}} \left( r - N - \frac{1}{2} \right)	, \quad & r \geq N+1, \\
		0 \leq g_{N}'(r) \leq C_{1}, \quad & r \geq 0, \\
		| g_{N}^{(k)}(r)| \leq C_{k}, \quad & r \geq 0, k \in \mathbb{N}.
	\end{cases}
\end{equation}

The terms $\bm{f}_{v} = \bm{f}_{v}(t,x,y(t))$ and $\bm{f}_{B} = \bm{f}_{B}(t,x,y(t))$ are forces acting on the fluid. The form of the noise term will be discussed below. For simplicity, we have set all the constants appearing in the equations to one. For the assumptions on the forces, see Section \ref{STMHD_ssec_NotAss}.

The stochastic MHD equations were first studied by S.S. Sritharan and P. Sundar in \cite{SS99} who proved existence of martingale solutions in the two- and three-dimensional case. The paper \cite{TWW16} by Z. Tan, D.H. Wang and H.Q. Wang contains results on existence of a global strong solution in 3D for small initial data as well as existence and uniqueness of a local solution. Their paper, however, was later retracted due to allegations of plagiarism.

Questions of ergodicity in two dimensions were studied by V. Barbu, G. Da Prato \cite{BDP07} for Wiener noise, and for $\alpha$-stable noise by T.L. Shen and J.H. Huang \cite{SH17}. K. Yamazaki \cite{Yamazaki19a} proved ergodicity in the case of random forcing only by a few modes. In three dimensions, he also proved ergodicity of a Faedo-Galerkin approximation of the MHD equations for degenerate noise in \cite{Yamazaki19b}.

The asymptotic behaviour of the SMHD equations was studied in the 2D additive white noise case by W.Q. Zhao and Y.R. Li \cite{ZL11a}, and in the 2D fractional case by J.H. Huang and T.L. Shen \cite{HS16}. H.Q. Wang \cite{WangH18} studied the system's exponential behaviour. Furthermore, S.H. Wang and Y.R. Li \cite{WL18} proved long-time robustness of the associated random attractor.

Jump-type and fractional noises have been studied by P. Sundar \cite{Sundar10}. U. Manna and M.T. Mohan \cite{MM13}, as well as E. Motyl \cite{Motyl14} studied the jump noise case, and the latter author provided a nice and very general framework for 3D hydrodynamic-type equations with L\'{e}vy noise on unbounded domains, generalising the 2D framework of I.D. Chueshov and A. Millet \cite{CM10}. Chueshov and Millet proved large deviations principles as well, and in \cite{CM11} also a Wong-Zakai approximation and a support theorem.

The existence of solutions to the non-resistive MHD equations with L\'{e}vy noise was investigated by U. Manna, M.T. Mohan and S.S. Sritharan \cite{MMS17}. 

K. Yamazaki proved existence of global martingale solutions for the nonhomogeneous system \cite{Yamazaki16}.
The case of non-Newtonian electrically conducting fluids and their long-time behaviour was studied in a paper by P.A. Razafimandimby and M. Sango \cite{RS15}.

In modelling the noise, we follow the approach of R. Mikulevicius and B.L. Rozovskii \cite{MR01, MR04, MR05} who proposed a multiplicative noise of transport-type for the Navier-Stokes equations, motivated by the turbulence theory of R.H. Kraichnan \cite{Kraichnan68}, which was further developed by K. Gawedzki and co-authors in \cite{GK96,GV00}. Transport-type noise was studied by several other authors as well, e.g. Z. Brze\'{z}niak, M. Capi\'{n}ski and F. Flandoli \cite{BCF91,BCF92}, as well as in Flandoli and D. Gatarek \cite[Section 3.3, pp. 378 f.]{FG95}. More recently, M. Hofmanov\'{a}, J.-M. Leahy and T. Nilssen \cite{HLN19} studied the problem via rough path methods.

We note that Mikulevicius and Rozovskii consider the case of only H\"older continuous $\sigma$ as being important to Kraichnan's turbulent velocity model, but for simplicity, we restrict ourselves to the case of differentiable $\sigma$ (see Assumption \ref{STMHD_prel_itm_ass_H2} below).

The tamed MHD equations were inspired by several works on the tamed Navier-Stokes equations. Existence and uniqueness as well as ergodicity for the stochastic tamed Navier-Stokes equations were studied by M. R\"ockner and X.C. Zhang in \cite{RZ09a}. The study of Freidlin-Wentzell-type large deviations was carried out by M. R\"ockner, T.S. Zhang and X.C. Zhang in \cite{RZZ10}. The case of existence, uniqueness and small time large deviation principles for the Dirichlet problem in bounded domains can be found in the work of M. R\"ockner and T.S. Zhang \cite{RZ12}. More recently, there has been resparked interest in the subject, with contributions by Z. Dong and R.R. Zhang \cite{DZ18} (existence and uniqueness for multiplicative L\'{e}vy noise), as well as Z. Brze\'{z}niak and G. Dhariwal \cite{BD19} (existence, uniqueness and existence of invariant measures in the full space $\R^{3}$ by different methods).

The deterministic tamed MHD equations were introduced by the author in \cite{Schenke20b}, where existence, uniqueness and regularity were studied.

From a physical point of view, the fact that our model is most appropriate for low to moderate Reynolds numbers, cf. \cite[Section 1.2.1]{Schenke20b}, raises the question of whether or not a stochastic model for turbulence (which is commonly associated with high Reynolds numbers) is appropriate in this setting. It is, nevertheless, an interesting mathematical problem that we want to address in this work. Furthermore, the randomness can also be seen as a model for other features of our system, including uncertainty.

\subsection{Damped Navier-Stokes Equations}
A related model for fluid flow through porous media is provided by the so-called \emph{(nonlinearly) damped Navier-Stokes equations}, sometimes called Brinkman-Forchheimer-extended Darcy Models. They are given by
\begin{align*}
	\frac{\partial \bm{v}}{\partial t} = \Delta \bm{v} - (\bm{v} \cdot \nabla)\bm{v} - \nabla p - \alpha |\bm{v}|^{\beta-1} \bm{v},
\end{align*}
where $\alpha > 0$ and $\beta \geq 1$. The damping term $- \alpha |\bm{v}|^{\beta-1} \bm{v}$ models the resistence to the motion of the flow resulting from physical effects like porous media flow, drag or friction, or other dissipative mechanisms (cf.  \cite{CJ08}). It represents a restoring force, which for $\beta = 1$ assumes the form of classical, linear damping, whereas $\beta > 1$ means a restoring force that grows superlinearly with the velocity (or magnetic field). X.J. Cai and Q.S. Jiu \cite{CJ08} first proved existence and uniqueness of a global strong solution for $\frac{7}{2} \leq \beta \leq 5$. 

This range was lowered down to $\beta \in (3,5]$ by Z.J. Zhang, X.L. Wu and M. Lu in \cite{ZWL11} who considered the case $\beta = 3$ to be critical \cite[Remark 3.1]{ZWL11}. Y. Zhou in \cite{Zhou12} proved the existence of a global solution for all $\beta \in [3,5]$. For the case $\beta \in [1,3)$, he established regularity criteria that ensure smoothness. Uniqueness holds for any $\beta \geq 1$ in the class of weak solutions.

The first problems studied in the stochastic damped Navier-Stokes case were related to the inviscid limit of the damped equations for $\beta = 1$ in 2D, cf. H. Bessaih and B. Ferrario \cite{BF14} and N. Glatt-Holtz, V. \v{S}ver\'{a}k and V. Vicol \cite{GHSV15}. B. You \cite{You17} proved existence of a random attractor under the assumption of well-posedness for additive noise for $\beta \in (3,5]$ (notably leaving out the critical, or tamed, case $\beta = 3$). K. Yamazaki \cite{Yamazaki18} proved a Lagrangian formulation and extended Kelvin's circulation theorem to the partially damped case (i.e. only a few components are damped, but the damping there is much stronger, e.g. $\beta_{k} = 9$ in two components $k=3,4$ in the 4D case). Z. Brze\'{z}niak and B. Ferrario \cite{BF19} showed existence of stationary solutions on the whole space $\R^{3}$ for $\beta = 1$. H. Liu and H.J. Gao \cite{LG18} proved existence and uniqueness of an invariant measure and a random attractor for $\beta \in [3,5]$. The same authors in \cite{GL19b}  proved a small-time large deviation principle for the same parameter range. Furthermore, again for the same range, for multiplicative noise, exponential convergence of the weak solutions in $L^{2}$ to the stationary solution as well as stabilisation were proved by H. Liu, L. Lin, C.F. Sun and Q.K. Xiao \cite{LLSX19}. Finally, the case of jump noise was treated by H. Liu and H.J. Gao \cite{GL19a}.

\subsection{Results and Structure of The Paper}\label{STMHD_ssec_results}
Here are the main results of the paper.  In the following, the underlying domain $\mbD \subset \R^{3}$ is assumed to be either $\R^{3}$ or the torus $\mbT^{3}$. Results apply to both cases unless otherwise stated.

First, we prove existence and uniqueness of a strong solution to the tamed MHD equations.

\begin{theorem}[Existence and uniqueness, Theorems \ref{STMHD_thm_uniq}, \ref{STMHD_thm_ex} and \ref{STMHD_thm_apriori} below]\label{STMHD_thm_exuniq}
	Let the coefficients $f := (\bm{f}_{v}, \bm{f}_{B})$, $\Sigma := (\bm{\sigma}, \bar{\bm{\sigma}})$ and $H := (\bm{h},\bar{\bm{h}})$ satisfy Assumptions \ref{STMHD_prel_itm_ass_H1}--\ref{STMHD_prel_itm_ass_H3} and let $y_{0} \in \mcH^{1}$. Then there exists a unique strong solution $y$ to the stochastic tamed MHD equations, in the sense of Definition \ref{STMHD_def_USS}, with the following properties:
	\begin{enumerate}[label=(\roman*), ref=(\roman*)]
		\item $y \in L^{2}(\O,P; C([0,T];\mcH^{1})) \cap L^{2}(\O,P; L^{2}([0,T];\mcH^{2}))$ for all $T > 0$ and
		\begin{equation}\label{STMHD_ex_eq_aprioriNbound}
			\E \left[ \sup_{t \in [0,T]} \| y(t) \|_{\mcH^{1}}^{2} \right] + \int\limits_{0}^{t} \E \left[ \| y(s) \|_{\mcH^{2}}^{2} \right] \diff s \leq C_{T,f,H} \left( 1 + \| y_{0} \|_{\mcH^{1}}^{2} \right) \cdot N.
		\end{equation}
		\item In $\mcH^{0}$, the following equation holds:
		\begin{equation}
			y(t) = \int\limits_{0}^{t} [ \mcA(y(s)) + \mcP f(s,y(s))] \diff s + \sum_{k=1}^{\infty} \int\limits_{0}^{t} \mcB_{k}(s,y(s)) \diff W_{s}^{k} \quad \forall t \geq 0, P-\mathrm{a.s.}
		\end{equation}
	\end{enumerate}
\end{theorem}
By the preceding theorem, for time-homogeneous coefficients, the solution process $y = \left( y(t;y_{0}) \right)_{t \geq 0}$ is a strong Markov process. For $t \geq 0$, we can define the associated Markov semigroup on the space $BC_{\text{loc}}(\mcH^{1})$ of bounded, locally uniformly continuous functions on the divergence-free Sobolev space $\mcH^{1}$, by
\begin{align*}
	T_{t} \phi(y_{0}) := \E \left[ \phi(y(t;y_{0})) \right], \quad \phi \in BC_{\text{loc}}(\mcH^{1}), y_{0} \in \mcH^{1}.
\end{align*}
Then, under a slightly stronger assumption on the coefficients of the noise, this semigroup satisfies the Feller property.
\begin{theorem}[Feller property, Theorem \ref{STMHD_thm_Feller} below]
	Under the Assumptions \ref{STMHD_prel_itm_ass_H1}, \ref{STMHD_prel_itm_ass_H2} and \ref{STMHD_FP_itm_ass_H3}, for every $t \geq 0$, $T_{t}$ maps $BC_{\text{loc}}(\mcH^{1})$ into itself, i.e. it is a Feller semigroup on $BC_{\text{loc}}(\mcH^{1})$.
\end{theorem}
In the case of a periodic domain $\mbD = \mbT^{3}$, we prove the existence of an invariant measure.
\begin{theorem}[Existence of an invariant measure, Theorem \ref{STMHD_thm_inv_ms_ex} below]
	Under the hypotheses \ref{STMHD_prel_itm_ass_H1}, \ref{STMHD_prel_itm_ass_H2}, \ref{STMHD_FP_itm_ass_H3}, in the periodic case $\mbD = \mbT^{3}$, there exists an invariant measure $\mu \in \mcP(\mcH^{1})$ associated to $(T_{t})_{t \geq 0}$, i.e. a measure $\mu$ such that for every $t \geq 0$, $\phi \in BC_{\text{loc}}(\mcH^{1})$
	\begin{align*}
		\int_{\mcH^{1}} T_{t} \phi(y_{0}) \diff \mu(y_{0}) = \int_{\mcH^{1}} \phi(y_{0}) \diff \mu(y_{0}).
	\end{align*}
\end{theorem}
The goal of this work was to generalise the well-posedness results of M. R\"ockner and X.C. Zhang \cite{RZ09b} for the stochastic tamed Navier-Stokes equations to the stochastic tamed MHD case. In doing so, we have to establish MHD versions of several  tools they use, which is technically more difficult.

To the best of the author's knowledge, neither the stochastic damped MHD equations nor the stochastic tamed MHD equations have not been considered so far. This work is a first step in a similar direction. 

The paper is organised as follows: we state our assumptions in Section \ref{STMHD_ssec_NotAss} as well as auxiliary results and estimates on the coefficients of our equation in Section \ref{STMHD_ssec_estimates_lemmas}. Since our proof of existence involves a tightness argument, we provide a tightness criterion in Section \ref{STMHD_sec_tightness}. 
Existence and uniqueness of a strong solution is then proved in Section \ref{STMHD_sec_exuniq}. We start by defining weak and strong solutions in \ref{STMHD_ssec_weakstrong}.
The next section, \ref{STMHD_ssec_uniq} is then devoted to proving pathwise uniqueness. 
Existence of a weak solution is proved in Section \ref{STMHD_ssec_exweak} via the by now classical strategy of proving \emph{a priori} estimates for the Faedo-Galerkin approximation to the equation and using them to infer tightness of the sequence of laws, Skorokhod coupling to obtain almost sure convergence, and concluding by proving uniform moment estimates and convergence in probability. This allows us to obtain -- using the Yamada-Watanabe theorem -- that there exists a unique strong solution. The Feller property of the semigroup as well as existence of invariant measures are then shown in Section \ref{STMHD_sec_Fellerinv}.

The results of this work were published as part of the author's PhD thesis \cite{Schenke20}. More detailed discussions of the literature as well as calculations can be found there.

\section{Preliminaries}\label{STMHD_sec_preliminaries}
In this section, we provide the basic tools needed later. After establishing our notation and the assumptions the coefficients, we prove some elementary estimates for the operators associated to the coefficients. We then provide a tightness criterion for later use.
\subsection{Notation and Assumptions}\label{STMHD_ssec_NotAss}

For a domain $G \subset \mathbb{R}^{3}$, we use the following notational hierarchy for $L^{p}$ and Sobolev spaces:  
\begin{enumerate}[label=(\arabic*), ref=(\arabic*)]
	\item The spaces $L^{p}(G,\mathbb{R})$ of real-valued integrable (equivalence classes of) functions -- e.g. the components $v^{i}$, $B^{i}$ -- will be denoted as $L^{p}(G)$ or $L^{p}$ if no confusion can arise.
	\item We may sometimes use the notation $\bm{L}^{p}(G) := L^{p}(G;\mathbb{R}^{3})$ to denote 3-D vector-valued integrable quantities, especially the velocity vector field and magnetic vector field $\bm{v}$ and $\bm{B}$.
	\item The divergence-free and $p$-integrable vector fields will be denoted by \verb=\mathbb= symbols, so $\mathbb{L}^{p}(G) := \bm{L}^{p}(G) \cap \divv^{-1} \{ 0 \}$. Its elements $\bm{v}$, $\bm{B}$ satisfy $\divv \bm{v} = \nabla \cdot \bm{v} = 0$, $\divv \bm{B} = 0$.
	\item Finally, we denote the space of the combined velocity and magnetic vector fields by \verb=\mathcal= symbols, i.e., $\mathcal{L}^{p}(G) := \mathbb{L}^{p}(G) \times \mathbb{L}^{p}(G)$. Its elements are of the form $y = (	\bm{v} , \bm{B})$.
\end{enumerate}
For Sobolev spaces, we use the same notational conventions, so for example $\mathbb{H}^{k}(G) := \bm{H}^{k}(G) \cap \divv^{-1} \{ 0 \} := W^{k,2}(G ; \mathbb{R}^{3}) \cap \divv^{-1} \{ 0 \} $ etc. Finally, if the domain of the functions is not in $\mathbb{R}^{3}$, in particular if it is a real interval (for the time variable), then we use the unchanged $L^{p}$ notation.
By $\ell^{2}$, we denote the space of square-summable sequences.

For brevity, we use the following terminology when discussing the terms on the right-hand side of the tamed MHD equations: the terms involving the Laplace operator are called the \emph{linear terms}, the terms involving the taming function $g_{N}$ are called \emph{taming terms} and the other terms are called the \emph{nonlinear terms}. Furthermore, we refer to the initial data $y_{0} = (\bm{v}_{0} , \bm{B}_{0} )$ and the force $f = ( \bm{f}_{v} , \bm{f}_{B})$ collectively as the \emph{data} of the problem.

For definiteness, we want to state the following very elementary relationship between the classical homogeneous second-order Sobolev norm and the norm we use in this work (which is defined via Bessel potentials).
\begin{lemma}\label{STMHD_thm_Besselpot_est}
	Let $\mbD \in \{ \R^{3}, \mbT^{3} \}$ and $y \in \mcH^{2}(\mbD)$. Then the following estimate holds:
	\begin{equation}\label{STMHD_prel_eq_Besselpot_est}
			\| y \|_{\dot{\mcW}^{2,2}(\mbD)}^{2} := \sum_{i,j=1}^{d} \| \partial_{x^{i}}\partial_{x^{j}} y \|_{\mcL^{2}(\mbD)}^{2} \leq 3^{2} \| y \|_{\mcH^{2}(\mbD)}^{2}.
	\end{equation}
\end{lemma}
\begin{proof}
Using Plancherel's theorem and Young's inequality, we find
\begin{align*}
	&\sum_{i,j=1}^{3} \| \partial_{x^{i}}\partial_{x^{j}} y \|_{\mcL^{2}}^{2} = \sum_{i,j=1}^{3} \| \xi_{i} \xi_{j} \hat{y} \|_{\mcL^{2}}^{2} = \sum_{i,j=1}^{3} \left( \| \xi_{i} \xi_{j} \hat{y} \|_{\mcL^{2}(\{ |\xi| \leq 1 \})}^{2} + \| \xi_{i} \xi_{j} \hat{y} \|_{\mcL^{2}(\{ |\xi| > 1 \})}^{2} \right) \\
	&\leq \sum_{i,j=1}^{3} \left( \| \hat{y} \|_{\mcL^{2}}^{2} + \left\| \frac{\left(\xi_{i}^{2} + \xi_{j}^{2} \right)}{2} \hat{y} \right\|_{\mcL^{2}(\{ |\xi| > 1 \})}^{2} \right) \leq 3^{2} \| y \|_{\mcL^{2}}^{2} + 3 \left\| | \xi |^{2} \hat{y} \right\|_{\mcL^{2}(\{ |\xi| > 1 \})}^{2}  \\
	&\leq 3^{2} \left( \| y \|_{\mcL^{2}}^{2} +  \| \Delta y \|_{\mcL^{2}}^{2}  \right) \leq 3^{2} \left( \| y \|_{\mcL^{2}}^{2} + 2\| \nabla y \|_{\mcL^{2}}^{2} +  \| \Delta y \|_{\mcL^{2}}^{2} \right) = 3^{2} \| (I-\Delta) y \|_{\mcL^{2}}^{2} = 3^{2} \| y \|_{\mcH^{2}}^{2}.
\end{align*}
\end{proof}

We will use the following estimate below.
\begin{lemma}
	For a function $y = (\bm{v}, \bm{B})$ such that $|\bm{v}| | \nabla \bm{v}| \in L^{2}$, $|\bm{B}| | \nabla \bm{B}| \in L^{2}$, it holds that $y \in \mcL^{12}$ and
	\begin{equation}\label{STMHD_eq_GNS_aux_est}
		\| y \|_{\mcL^{12}}^{4} \leq C \left( \left\| |\bm{v}| |\nabla \bm{v}| \right\|_{L^{2}}^{2} + \left\| |\bm{B}| |\nabla \bm{B}| \right\|_{L^{2}}^{2} \right).
	\end{equation}
\end{lemma}
\begin{proof}
By the Gagliardo-Nirenberg-Sobolev inequality for scalar functions $f \in W^{1}(\R^{3})$
\begin{align*}
	\| f \|_{L^{6}} \leq C \| \nabla f \|_{L^{2}},
\end{align*}
it follows that
\begin{align*}
	\| y \|_{\mcL^{12}}^{4} = \| |y|^{2} \|_{\mcL^{6}}^{2} &\leq C \| \nabla |y|^{2} \|_{L^{2}}^{2} = 4C \int \sum_{i=1}^{3} \left| \sum_{j=1}^{3}  \left( v_{j} \partial_{i} v_{j} +  B_{j} \partial_{i} B_{j} \right) \right|^{2}  \diff x \\
	&\leq 24C \int \sum_{j=1}^{3}  \left( |v_{j}|^{2} \sum_{i=1}^{3} | \partial_{i} v_{j}|^{2} +  |B_{j}| \sum_{i=1}^{3} |\partial_{i} B_{j}|^{2} \right)   \diff x \\
	&\leq 24C \left( \left\| |\bm{v}| |\nabla \bm{v}| \right\|_{L^{2}}^{2} + \left\| |\bm{B}| |\nabla \bm{B}| \right\|_{L^{2}}^{2} \right),
\end{align*}
which proves the claim.
\end{proof} 

We make the following assumptions on our coefficients:
\begin{enumerate}[label=(H\arabic*), ref=(H\arabic*)]
	\item\label{STMHD_prel_itm_ass_H1} For any $T > 0$, the function $f = \begin{pmatrix} \bm{f}_{v}, \bm{f}_{B} \end{pmatrix}$ with
	 $\bm{f}_{v}, \bm{f}_{B} \colon [0,T] \times \mbD \times \R^{6} \rightarrow \R^{3} $ satisfies: there is a constant $C_{T,f} > 0$ and a function $F_{f}(t,x) \in L^{\infty}([0,T];L^{1}(\mbD))$
	\begin{align*}
	 	| \partial_{x^{j}} f(t,x,y)|^{2} + | f(t,x,y)|^{2} &\leq C_{T,f} |y|^{2} + F_{f}(t,x), \quad j = 1,2, 3, x \in \mbD, y \in \R^{6}, \\
	 	| \partial_{y^{l}} f(t,x,y) | &\leq C_{T,f}, \quad l = 1,\ldots, 6, x \in \mbD, y \in \R^{6}.
	\end{align*}
	\item\label{STMHD_prel_itm_ass_H2} For any $T > 0$, for the function $\Sigma = \begin{pmatrix} \bm{\sigma}, \bar{\bm{\sigma}}\end{pmatrix}$ with $\bm{\sigma}$, $\bar{\bm{\sigma}} \colon [0,T] \times \mbD \rightarrow \ell^{2}(\R^{3})$, there are constants $C_{\sigma,T}, C_{\bar{\sigma},T} > 0$ such that for $j=1,2,3$
	\begin{align*}
		\sup_{t \in [0,T], x \in \mbD} \| \partial_{x^{j}} \bm{\sigma}(t,x) \|_{\ell^{2}} \leq C_{\sigma,T}, \quad \sup_{t \in [0,T], x \in \mbD} \| \partial_{x^{j}} \bar{\bm{\sigma}}(t,x) \|_{\ell^{2}} \leq C_{\bar{\sigma},T},
	\end{align*}
	as well as
	\begin{align*}
		\sup_{t \in [0,T], x \in \mbD} \| \Sigma(t,x) \|_{\ell^{2}}^{2} \leq \frac{1}{36}.
	\end{align*}
	\item\label{STMHD_prel_itm_ass_H3} For any $T > 0$, for the function $H = (\bm{h}, \bar{\bm{h}})$ with $\bm{h}, \bar{\bm{h}} \colon [0,T] \times \mbD \times \R^{6} \rightarrow  \ell^{2}(\R^{3})$, there exists a constant $C_{T,H} > 0$ and $F_{H}(t,x) \in L^{\infty}([0,T];L^{1}(\mbD))$ such that for any $0 \leq t \leq T$, $x \in \mbD$, $y \in \R^{6}$ and $j=1,2,3$, $l = 1,2, \ldots, 6$
	\begin{align*}
		\| \partial_{x^{j}} H(t,x,y) \|_{\ell^{2}}^{2} + \| H(t,x,y) \|_{\ell^{2}}^{2} &\leq C_{T,H} | y |^{2} + F_{H}(t,x), \\
		\| \partial_{y^{l}} H(t,x,y) \|_{\ell^{2}} &\leq C_{T,H}.
	\end{align*}
\end{enumerate}
\begin{remark}
	The origin of the constant $36 = 4 \cdot 3^{2}$ in Assumption \ref{STMHD_prel_itm_ass_H2} lies in the fact that in the places where we need the numerical value, i.e. in the proof of \eqref{STMHD_prel_eq_B_L2l2H1} of Lemma \ref{STMHD_thm_estB} as well as in the proof of Lemma \ref{STMHD_FP_thm_Lemma}, we estimate the homogeneous second-order Sobolev norm against the Bessel potential norm via Lemma \ref{STMHD_thm_Besselpot_est} below, giving an additional factor of $9 = 3^2$. We do not claim that this value for $\Sigma$ is sharp, but we wanted to give an explicit bound that suffices to make all the calculations work.
	
	The integrability conditions $F_{f}, F_{H} \in L^{\infty}([0,T];L^{1}(\mbD))$ are used in proving continuity in time, as we want to estimate
	$$
		\int_{s}^{t} \| F_{f,H}(r) \|_{L^{1}(\mbD)} \diff  r \leq C |t-s|
	$$
in the proof of Lemma \ref{STMHD_thm_tightness}.

One could also model the equations in a way that the terms $\bm{f}_{v}$ and $\bm{h}$ in the equations of $\bm{v}$ depend only on $\bm{v}$ instead of $y$. However, this case is included in our assumptions, which are more symmetric this way.
\end{remark}

We define the set of solenoidal test functions as
\begin{align*}
	\mcV := \{ y = \begin{pmatrix}
	\bm{v}, \bm{B}
\end{pmatrix} ~|~ \bm{v}, \bm{B} \in C_{c}^{\infty}(\mbD;\R^{3}), \divv(\bm{v}) = \divv(\bm{B}) = 0 \}.
\end{align*}
As in \cite[Lemma 2.1]{RZ09a}, $\mcV$ is dense in $\mcH^{k}$ for any $k \in \N$.
Let $P \colon L^{2}(\mbD;\R^{3}) \rightarrow \mbH^{0}$ be the Leray-Helmholtz projection. In the case of $\mbD \in \{ \mbT, \R^{3} \}$, $P$ commutes with derivative operators (cf. \cite[Lemma 2.9, p. 52]{RRS16}) and can be restricted to a bounded linear operator
\begin{align*}
	P \restr{H^{m}} \colon H^{m} \rightarrow \mbH^{m}.
\end{align*} 
Furthermore, consider the tensorised projection
\begin{align*}
	\mathcal{P} := P \otimes P, \quad \mcP y := (P \otimes P) \begin{pmatrix} \bm{v} \\ \bm{B}	\end{pmatrix} = \begin{pmatrix} P \bm{v} \\ P \bm{B} \end{pmatrix}.
\end{align*}
Then $\mcP \colon \mcL^{2} \rightarrow \mcH^{0}$ is a bounded linear operator:
\begin{align*}
	\| \mcP y \|_{\mcH^{0}}^{2} = \| P \bm{v} \|_{\mbH^{0}}^{2} + \| P \bm{B} \|_{\mbH^{0}}^{2} \leq \| P \|_{L^{2} \rightarrow \mbH^{0}}^{2} \left( \| \bm{v} \|_{L^{2}}^{2} + \| B \|_{L^{2}}^{2} \right) = \| P \|_{L^{2} \rightarrow \mbH^{0}}^{2} \| y \|_{\mathcal{L}^{2}}^{2}.
\end{align*}
We now define operators
\begin{align*}
	\mcA(y) := \mcP \D y - \mcP \begin{pmatrix}
		(\bm{v} \cdot \nabla) \bm{v}  - (\bm{B} \cdot \nabla) \bm{B} \\ (\bm{v} \cdot \nabla) \bm{B} - (\bm{B} \cdot \nabla) \bm{v} 
	\end{pmatrix} - \mcP \left( g_{N}(|y|^{2})y \right),
\end{align*}
\begin{equation}\label{STMHD_prel_eq_def_llbracket}
	\langle \mcA(y), \tilde{y} \rangle_{\mcH^{1}} = \left\langle \mcA(y), (I - \D) \tilde{y} \right\rangle_{\mcH^{0}} = \mcA_{1}(y,\tilde{y}) + \mcA_{2}(y,\tilde{y}) + \mcA_{3}(y,\tilde{y}),
\end{equation}
where
\begin{align*}
	\mcA_{1}(y,\tilde{y}) &:= \left\langle \mcP \D y, (I - \D) \tilde{y} \right\rangle_{\mcH^{0}}, \\
	\mcA_{2}(y,\tilde{y}) &:= -\left\langle \mcP \begin{pmatrix}
		(\bm{v} \cdot \nabla) \bm{v}  - (\bm{B} \cdot \nabla) \bm{B} \\ (\bm{v} \cdot \nabla) \bm{B} - (\bm{B} \cdot \nabla) \bm{v} 
	\end{pmatrix}, (I - \D) \tilde{y} \right\rangle_{\mcH^{0}}, \\
	\mcA_{3}(y,\tilde{y}) &:= - \left\langle \mcP g_{N}(|y|^{2}) y, (I - \D) \tilde{y} \right\rangle_{\mcH^{0}}.
\end{align*}
For the noise terms, we define $\Sigma(t,x) := \begin{pmatrix}
	\bm{\sigma}_{k}(t,x) \\ \bar{\bm{\sigma}}_{k}(t,x)
\end{pmatrix}_{k \in \N} \in \ell^{2}(\R^{6})$ and for $y = \begin{pmatrix}
	\bm{v} \\ \bm{B} \end{pmatrix}$
\begin{align*}
	\left( \Sigma_{k}(t,x) \cdot \nabla \right) y := \begin{pmatrix}
	(\bm{\sigma}_{k}(t,x) \cdot \nabla ) \bm{v} \\ (\bar{\bm{\sigma}}_{k}(t,x) \cdot \nabla) \bm{B}
\end{pmatrix}.
\end{align*}
Similarly, we define $H_{k}(t,y) := \begin{pmatrix}
	\bm{h}_{k}(t,y) \\ \bar{\bm{h}}_{k}(t,y)
\end{pmatrix}$, and
\begin{align*}
	\mcB_{k}(t,x,y) := \mcP \left( (\Sigma_{k}(t,x) \cdot \nabla) y \right) + \mcP H_{k}(t,x,y).
\end{align*}
Finally, let $\{ W_{t}^{k} ~|~ t \in \R_{+}, k \in \N \}, \{ \bar{W}_{t}^{k} ~|~ t \in \R_{+}, k \in \N \}$ be two independent sequences of independent Brownian motions, and define
\begin{align*}
	\mcW_{t}^{k} := \begin{pmatrix}
			W_{t}^{k} \\ \bar{W}_{t}^{k}
	\end{pmatrix}.
\end{align*}
The stochastic integral in our equations is defined by
\begin{align*}
	 \int\limits_{0}^{t} \mcB_{k}(s,x,y) \diff \mcW_{s}^{k} := \begin{pmatrix}
	 	\int\limits_{0}^{t} P (\bm{\sigma}_{k}(s,x) \cdot \nabla )\bm{v} + P \bm{h}_{k}(s,x,y) \diff W_{s}^{k} \\
	 	\int\limits_{0}^{t} P (\bar{\bm{\sigma}}_{k}(s,x) \cdot \nabla )\bm{B} + P \bar{\bm{h}}_{k}(s,x,y) \diff \bar{W}_{s}^{k}
	 \end{pmatrix}.
\end{align*}
The Brownian motions $W$ and $\bar{W}$ can be understood as independent cylindrical Brownian motions on the space $\ell^{2}$. Similarly, $\mcW$ is a cylindrical Brownian motion on the space $\ell^{2} \times \ell^{2}$. For $y \in \mcH^{m}$, $m=1,2$, $\mcB(t,x,y(x))$ can be understood as a linear operator
\begin{align*}
	\mcB(t,\cdot,y) \colon \ell^{2} \times \ell^{2} \rightarrow \mcH^{m-1}.
\end{align*}
To make this clear, we note that if we take the canonical basis of $\ell^{2}$, i.e. orthonormal basis consisting of the sequences $e_{1} := (1,0,0, \ldots), e_{2} := (0,1,0,0,\ldots), \ldots$, then the system 
\begin{align*}
	\left\{ \begin{pmatrix}
		e_{k} \\ 0
	\end{pmatrix} \in \ell^{2} \times \ell^{2} ~|~ k \in \N \right\} \cup \left\{
	\begin{pmatrix}
		0 \\ e_{k}
	\end{pmatrix} \in \ell^{2} \times \ell^{2} ~|~ k \in \N \right\}
\end{align*}
forms an orthonormal basis of $\ell^{2} \times \ell^{2}$. Then we define 
\begin{align*}
	\mcB(t,\cdot,y) \begin{pmatrix}
		e_{k} \\ 0
	\end{pmatrix}(x) := \begin{pmatrix}
		P (\bm{\sigma}_{k}(t,x) \cdot \nabla) \bm{v}(x) + P \bm{h}_{k}(t,x,y(x)) \\ 0
	\end{pmatrix} \in \mcH^{m-1},
\end{align*}
and
\begin{equation}\label{STMHD_eq_ONB_ell2}
	\mcB(t,\cdot,y) \begin{pmatrix}
		0 \\ e_{k}
	\end{pmatrix}(x) := \begin{pmatrix}
		0 \\ P (\bar{\bm{\sigma}}_{k}(t,x) \cdot \nabla) \bm{v}(x) + P \bar{\bm{h}}_{k}(t,x,y(x)) 
	\end{pmatrix} \in \mcH^{m-1}.
\end{equation}
It turns out that $\mcB$ is even a Hilbert-Schmidt operator, i.e. $\mcB(t,\cdot,y) \in L_{2}(\ell^{2} \times \ell^{2} ; \mcH^{m-1})$, as will be proved below in Lemma \ref{STMHD_thm_estB}. Hence we are in the usual framework of stochastic analysis on Hilbert spaces, cf. \cite{DPZ92} or \cite{LR15}.

We can now formulate Equation \eqref{STMHD_eq_STMHD} in the following abstract form as an evolution equation:
\begin{equation}\label{STMHD_eq_evol_eqn}
	\begin{cases}
		\diff y(t) &= [\mcA(y(t)) + \mcP f(t,y(t))]\diff t + \sum_{k=1}^{\infty} \mcB_{k}(t,y(t)) \diff \mcW_{t}^{k}, \\
		y(0) &= y_{0} \in \mcH^{1}.
	\end{cases}
\end{equation}
\subsection{Estimates on the Operators \texorpdfstring{$\mathcal{A}$}{A} and \texorpdfstring{$\mathcal{B}$}{B}}\label{STMHD_ssec_estimates_lemmas}
In this section, we will collect important but elementary estimates on the operators $\mcA$ and $\mcB$. These play an important role in deriving the \emph{a priori} estimates for the STMHD equations. The first lemma is concerned with estimates for the case of testing with the solution $y$ itself.
\begin{lemma}[Estimates for $\mcA$]\label{STMHD_thm_estA}
	Let $y \in \mcH^{2}$. Then the following estimates hold true
	\begin{equation}\label{STMHD_prel_eq_A_bdd}
		\| \mcA(y) \|_{\mcH^{0}} \leq C \left( 1 + \|y\|_{\mcH^{0}}^{6} + \| y \|_{\mcH^{2}}^{2} \right),
	\end{equation}
	\begin{align}
		\label{STMHD_prel_eq_A_testing1} \langle \mcA(y), y \rangle_{\mcH^{0}} &= - \| \nabla y \|_{\mcH^{0}}^{2} - \| \sqrt{g_{N}(|y|^{2})} |y| \|_{\mcL^{2}}^{2} \\
		\label{STMHD_prel_eq_A_testing2} 
		&\leq - \| \nabla y \|_{\mcH^{0}}^{2} - \| y \|_{\mcL^{4}}^{2} + CN \| y \|_{\mcH^{0}}^{2},
	\end{align}
	\begin{equation}\label{STMHD_prel_eq_A_H1_testing}
		\begin{split}
		\langle \mcA(y),y \rangle_{\mcH^{1}} &\leq - \frac{1}{2} \| y \|_{\mcH^{2}}^{2} - \Big( \| |\bm{v}| \cdot |\nabla \bm{v} | \|_{L^{2}}^{2} + \| |\bm{B}| \cdot |\nabla \bm{B} | \|_{L^{2}}^{2} \\
		&\quad + \| |\bm{v}| \cdot |\nabla \bm{B} | \|_{L^{2}}^{2} + \| |\bm{B}| \cdot |\nabla \bm{v} | \|_{L^{2}}^{2} \Big) \\
		&\quad + (2N+1) \| \nabla y \|_{\mcH^{0}}^{2} + \| y \|_{\mcH^{0}}^{2}.
		\end{split}
	\end{equation}
\end{lemma}
\begin{proof}
	This follows in the same way as \cite[Lemma 2.2, p. 9]{Schenke20b}.
\end{proof}
The next lemma deals with the case of testing the right-hand side of our equation with another, arbitrary, compactly supported test function.
\begin{lemma}[Estimates for $\mcA$ and $\mcB$]\label{STMHD_thm_estAB}
	Let $\tilde{y} \in \mcV$ with compact support $\supp( \tilde{y}) \subset \mcO := \{ x \in \mbD ~|~ |x| \leq m \}$ for some $m \in \N$. Let $T>0$. Then for any $y, y' \in \mcH^{2}$ and $t \in [0,T]$, we have
	\begin{align}
		\label{STMHD_prel_eq_A_L3_est}| \langle \mcA(y), \tilde{y} \rangle_{\mcH^{1}} | &\leq C_{\tilde{y}} \left( 1 + \| y \|_{L^{3}(\mcO)}^{3} \right), \\
		\label{STMHD_prel_eq_B_H1_est} \| \langle \mcB_{\cdot}(t,y), \tilde{y} \rangle_{\mcH^{1}} \|_{\ell^{2}}^{2} &\leq C_{\tilde{y},\Sigma, H, T} \left(  \| F_{H}(t) \|_{L^{1}(\mbD)} + \| y \|_{L^{2}(\mcO)}^{2}  \right),
	\end{align}
	and
	\begin{equation}\label{STMHD_prel_eq_stability_est}
		| \langle \mcA(y) - \mcA(y'), \tilde{y} \rangle_{\mcH^{1}}| \leq C_{\tilde{y},N} \| y - y' \|_{\mcL^{2}}\left( 1 + \| y \|_{\mcH^{1}}^{2} + \| y' \|_{\mcH^{1}}^{2}  \right).
	\end{equation}
\end{lemma}
\begin{proof}
	The first inequality follows easily from the following calculations:
	\begin{align*}
		\mcA_{1}(y,\tilde{y}) &= \langle y, (I-\D) \D \tilde{y} \rangle_{\mcH^{0}} \leq \| y \|_{L^{2}(\mcO)} \| (I-\D)\D \tilde{y} \|_{\mcH^{0}} \leq C_{\tilde{y}} \| y \|_{L^{3}(\mcO)} \| \tilde{y} \|_{\mcH^{4}},
	\end{align*}
	where the constant $C_{\tilde{y}}$ depends on the domain $\mcO$, and hence on $\tilde{y}$. For the next term we find
	\begin{align*}
		\mcA_{2}(y,\tilde{y}) &\leq \Big( \left( \| \bm{v} \|_{L^{2}(\mcO)}^{2} + \| \bm{B} \|_{L^{2}(\mcO)}^{2} \right) \cdot \sup_{x \in \mbD} |\nabla (I-\D) \tilde{\bm{v}}(x)| \\
		&\quad + 2 \| \bm{v} \|_{L^{2}(\mcO)} \| \bm{B} \|_{L^{2}(\mcO)}  \cdot \sup_{x \in \mbD} |\nabla (I-\D) \tilde{\bm{B}}(x)| \Big) \\
		&\leq 2 \| y \|_{\mcL^{2}(\mcO)}^{2} \cdot \sup_{x \in \mbD} |\nabla (I-\D) \tilde{\bm{y}}(x)| \\
		&\leq C_{\tilde{y}} \| y \|_{\mcL^{3}(\mcO)}^{2} \cdot \sup_{x \in \mbD} |\nabla (I-\D) \tilde{\bm{y}}(x)|.
	\end{align*}
	Finally, 
	\begin{align*}
		\mcA_{3}(y,\tilde{y})  \leq \| g_{N}(|y|^{2})y \|_{\mcL^{1}(\mcO)} \cdot \sup_{x \in \mbD} |(I-\D)\tilde{y}(x) |\leq C \| y \|_{\mcL^{3}(\mcO)}^{3} \sup_{x \in \mbD} |(I-\D)\tilde{y}(x) |.
	\end{align*}
	Combining the above estimates yields \eqref{STMHD_prel_eq_A_L3_est}.
	
	For the second estimate \eqref{STMHD_prel_eq_B_H1_est}, we have by the boundedness of the Leray-Helmholtz projections
	\begin{align*}
		&\| \langle \mcB_{\cdot}(t,y), \tilde{y} \rangle_{\mcH^{1}} \|_{\ell^{2}}^{2} \leq \left\| \langle (\Sigma \cdot \nabla) y, (I-\D)\tilde{y} \rangle_{\mcL^{2}} + \langle H(t,y),  (I-\D) \tilde{y} \rangle_{\mcL^{2}} \right\|_{\ell^{2}}^{2} \\
		&\leq 2 \Big( \sum_{k} \left| \left\langle (\Sigma_{k} \cdot \nabla) y, (I-\D) \tilde{y} \right\rangle_{\mcL^{2}} \right|^{2} + \left| \left\langle H_{k}(t,y),  (I-\D) \tilde{y} \right\rangle_{\mcL^{2}} \right|^{2} \Big).
	\end{align*}
	The first term can be estimated using the definitions of the scalar products and norms involved, integration by parts, the product rule and Jensen's inequality:
	\begin{align*}
		&\sum_{k} \left| \left\langle (\Sigma_{k} \cdot \nabla) y, (I-\D) \tilde{y} \right\rangle_{\mcL^{2}} \right|^{2} \\
&= \sum_{k} \Bigg| \int_{\mcO} \sum_{j,l=1}^{3} v^{l}  \left( \partial_{x^{j}}  \sigma_{k}^{j} \right) (I-\D) \tilde{v}^{l} + v^{l}   \sigma_{k}^{j}  (I-\D) \partial_{x^{j}} \tilde{v}^{l} \\
&\quad \quad +  B^{l}  \left( \partial_{x^{j}}  \bar{\sigma}_{k}^{j} \right) (I-\D) \tilde{B}^{l} + B^{l}   \bar{\sigma}_{k}^{j}  (I-\D) \partial_{x^{j}} \tilde{B}^{l} \diff x \Bigg|^{2} \\
&\leq 36 \lambda(\mcO) \sum_{k}  \int_{\mcO} \sum_{j,l=1}^{3} \left| v^{l}  \left( \partial_{x^{j}}  \sigma_{k}^{j} \right) (I-\D) \tilde{v}^{l} \right|^{2} + \left| v^{l}   \sigma_{k}^{j}  (I-\D) \partial_{x^{j}} \tilde{v}^{l} \right|^{2} \\
&\quad \quad +  \left| B^{l}  \left( \partial_{x^{j}}  \bar{\sigma}_{k}^{j} \right) (I-\D) \tilde{B}^{l} \right|^{2} + \left| B^{l}   \bar{\sigma}_{k}^{j}  (I-\D) \partial_{x^{j}} \tilde{B}^{l} \right|^{2} \diff x \\ 
&\leq 36 \lambda(\mcO) \bigg( \sup_{t \in [0,T], x \in \mbD} \| \nabla_{x} \Sigma_{\cdot}(t,x) \|_{\ell^{2}}^{2} \sup_{x \in \mbD} |(I-\D) \tilde{y}(x) |^{2} \\
&\quad + \sup_{t \in [0,T], x \in \mbD} \| \Sigma_{\cdot}(t,x) \|_{\ell^{2}}^{2} \sup_{x \in \mbD} |\nabla (I-\D) \tilde{y}(x) |^{2} \bigg) \| y \|_{\mcL^{2}(\mcO)}^{2}.
	\end{align*}
	We estimate this further by using Assumption \ref{STMHD_prel_itm_ass_H2}, arriving at
	\begin{align*}
	\sum_{k} \left| \left\langle (\Sigma_{k} \cdot \nabla) y, (I-\D) \tilde{y} \right\rangle_{\mcL^{2}} \right|^{2} \leq C_{\Sigma, T,\tilde{y}} \| y \|_{\mcL^{2}(\mcO)}^{2}.
	\end{align*}
	Similarly, for the second term we find ($\lambda$ denotes the Lebesgue measure)
	\begin{align*}
		&\sum_{k} \left| \left\langle H_{k}(t,y),  (I-\D) \tilde{y} \right\rangle_{\mcL^{2}} \right|^{2} \\
		&\leq 2 \lambda(O) \sum_{k} \int_{\mcO}
			\left| \langle \bm{h}_{k}(t,x,y) , (I-\D) \tilde{\bm{v}}(x) \rangle_{\R^{3}} \right|^{2} + \left| \langle \bar{\bm{h}}_{k}(t,x,y), (I - \D) \tilde{\bm{B}}(x) \rangle_{\R^{3}} \right|^{2} \diff x \\
		&\leq 2 \lambda(O) \sup_{x \in \mbD} |(I-\D) \tilde{\bm{v}}(x)|^{2} \int_{\mcO} \| \bm{h}_{\cdot}(t,x,y) \||_{\ell^{2}}^{2} + \sup_{x \in \mbD} |(I-\D) \tilde{\bm{B}}(x)|^{2} \| \bar{\bm{h}}_{\cdot}(t,x,y) \|_{\ell^{2}}^{2} \diff x \\
		&\leq 2 \lambda(O)C_{T,H}  \sup_{x \in \mbD} |(I-\D) \tilde{y}(x)|^{2} \int_{\mcO} |y(x)|^{2} + F_{H}(t,x) \diff x \\
		&\leq C_{H,T,\tilde{y}} \left( \| y \|_{\mcL^{2}(\mcO)}^{2} + \| F_{H}(t) \|_{L^{1}(\mbD)} \right).
	\end{align*}
	Thus, altogether we have proved \eqref{STMHD_prel_eq_B_H1_est}.
	
	Estimate \eqref{STMHD_prel_eq_stability_est} follows in the same way as \cite[Lemma 2.4, Equation $(2.13)$, pp. 218 f.]{RZ09b}.
\end{proof}
\begin{lemma}[Estimates for $\mcB$]\label{STMHD_thm_estB}
 For any $T > 0$, $0 \leq t \leq T$ and $y \in \mcH^{2}$, the following estimates hold:
 \begin{align}
 	\label{STMHD_prel_eq_B_L2l2H0} \| \mcB(t,y) \|_{L_{2}(\ell^{2} \times \ell^{2};\mcH^{0})}^{2} &\leq \frac{1}{2} \| y \|_{\mcH^{1}}^{2} + C_{T,H} \| y \|_{\mcH^{0}}^{2} + \| F_{H}(t) \|_{L^{1}(\mbD)}, \\
 	\label{STMHD_prel_eq_B_L2l2H1} \| \mcB(t,y) \|_{L_{2}(\ell^{2}\times \ell^{2};\mcH^{1})}^{2} &\leq \frac{1}{2} \| y \|_{\mcH^{2}}^{2} + C_{T,H,\Sigma} \| y \|_{\mcH^{1}}^{2} + C \| F_{H}(t) \|_{L^{1}(\mbD)}.
 \end{align}
\end{lemma}
\begin{proof}
	To calculate the Hilbert-Schmidt norm (cf. \cite[Definition B.0.5, p. 217]{LR15}), we take the orthonormal basis of $\ell^{2} \times \ell^{2}$ from Equation \eqref{STMHD_eq_ONB_ell2} and enumerate it as a set $\{ \hat{e}_{l} \}_{l \in \N}$. Then by definition, for $m=0,1$,
	\begin{align*}
		\| \mcB(t,y) \|_{L_{2}(\ell^{2} \times \ell^{2};\mcH^{m})}^{2} &= \sum_{l \in \N} \| \mcB(t,y)  \hat{e}_{l} \|_{\mcH^{m}}^{2} = \sum_{k \in \N} \left\| \mcB(t,y)  \begin{pmatrix} e_{k} \\ 0\end{pmatrix} \right\|_{\mcH^{m}}^{2} + \sum_{k \in \N} \left\| \mcB(t,y)  \begin{pmatrix} 0 \\  e_{k}\end{pmatrix} \right\|_{\mcH^{m}}^{2} \\
		&= \sum_{k \in \N} \| \mcB_{k}(t,y)  \|_{\mcH^{m}}^{2}.
	\end{align*}
The first inequality follows from Assumptions \ref{STMHD_prel_itm_ass_H2} and \ref{STMHD_prel_itm_ass_H3}:
	\begin{align*}
		&\| \mcB(t,y) \|_{L_{2}(\ell^{2} \times \ell^{2};\mcH^{0})}^{2} = \sum_{k} \int_{\mbD} |\mcB_{k}(t,x,y)|^{2} \diff x \\
&\leq 2 \int_{\mbD}  \|\Sigma(t,x)  \|_{\ell^{2}}^{2} | \nabla y |^{2} + \| H_{k}(t,x,y) \|_{\ell^{2}}^{2} \diff x \\
&\leq 2 \left( \sup_{t \in [0,T], x \in \mbD} \|\Sigma(t,x)  \|_{\ell^{2}}^{2} \right)
\int_{\mbD}   | \nabla y |^{2} + 2 C_{T,H} |y|^{2} + F_{H}(t,x) \diff x \\
&\leq \frac{1}{2}
 \| y \|_{\mcH^{1}}^{2} + (2C_{T,H}-\frac{1}{2}) \| y \|_{\mcH^{0}}^{2} + \| F_{H}(t) \|_{L^{1}(\mbD)}.
	\end{align*}
	For the second inequality, we note that
	\begin{align*}
		\| \mcB(t,y) \|_{L_{2}(\ell^{2}\times \ell^{2}; \mcH^{1})}^{2} = \| \mcB(t,y) \|_{L_{2}(\ell^{2}\times \ell^{2}; \mcH^{0})}^{2} + \| \nabla \mcB(t,y) \|_{L_{2}(\ell^{2}\times \ell^{2}; \mcH^{0})}^{2},
	\end{align*}
	and, using the commutativity of derivatives and $\mcP$ as well as the chain rule, we find
	\begin{align*}
		&\partial_{x^{j}} \mcB_{k}(t,x,y) = 	\partial_{x^{j}} \mcP\left[  (\Sigma_{k}(t,x) \cdot \nabla )y + H_{k}(t,x,y) \right] = \mcP \partial_{x^{j}}\left[  (\Sigma_{k}(t,x) \cdot \nabla )y + H_{k}(t,x,y) \right] \\
		 &= \mcP \big[  (\partial_{x^{j}} \Sigma_{k}(t,x) \cdot \nabla )y + ( \Sigma_{k}(t,x) \cdot \nabla )\partial_{x^{j}}y + \left(\partial_{x^{j}} H_{k} \right)(t,x,y) + \sum_{i=1}^{6} \partial_{y^{i}} H_{k}(t,x,y) \partial_{x^{j}} y^{i} \big].
	\end{align*}
	Therefore, employing Assumptions \ref{STMHD_prel_itm_ass_H2}, \ref{STMHD_prel_itm_ass_H3}, as well as Equation \eqref{STMHD_prel_eq_Besselpot_est}, we find
	\begin{align*}
		&\| \nabla \mcB(t,y) \|_{L_{2}(\ell^{2}\times \ell^{2}; \mcH^{0})}^{2} \leq  \sum_{k} \sum_{j=1}^{3} \int_{\mbD} 2|( \Sigma_{k}(t,x) \cdot \nabla )\partial_{x^{j}}y|^{2} + 2 \big| (\partial_{x^{j}} \Sigma_{k}(t,x) \cdot \nabla )y +  \\
		&\quad + \left(\partial_{x^{j}} H_{k} \right)(t,x,y) + \sum_{i=1}^{6} \partial_{y^{i}} H_{k}(t,x,y) \partial_{x^{j}} y^{i} \big|^{2} \diff x \\
		&\leq 2 \int_{\mbD} \sum_{i,j=1}^{3} \| \Sigma(t,x) \|_{\ell^{2}}^{2} | \partial_{x^{i}} \partial_{x^{j}}y |^{2} \diff x + 6 \int_{\mbD} \sum_{j=1}^{3} \| \partial_{x^{j}} \Sigma(t,x)\|_{\ell^{2}}^{2} |\nabla y|^{2} \\
		&\quad + \|\left(\partial_{x^{j}} H_{k} \right)(t,x,y) \|_{\ell^{2}}^{2} + 6 \sum_{i=1}^{6} \|\partial_{y^{i}} H(t,x,y) \|_{\ell^{2}}^{2} | \partial_{x^{j}} y^{i} |^{2} \diff x \\
		&\leq 2 \cdot 3^{2} \sup_{t \in [0,T], x \in \mbD} \| \Sigma(t,x) \|_{\ell^{2}}^{2} \| y \|_{\mcH^{2}}^{2} +  6 \Big( \sup_{t \in [0,T], x \in \mbD} \| \partial_{x^{j}} \Sigma(t,x)\|_{\ell^{2}}^{2} \|\nabla y\|_{\mcH^{0}}^{2}  \\
		&\quad + C_{T,H} \| y \|_{\mcH^{0}}^{2} + \| F_{H}(t) \|_{L^{1}(\mbD)} + C_{T,H} \| \nabla y \|_{\mcH^{0}}^{2}  \Big) 		\\
		&\leq  \frac{1}{2}\| y \|_{\mcH^{2}}^{2} + C_{T,H,\Sigma} \left( \| \nabla y \|_{\mcH^{0}}^{2} + \| y \|_{\mcH^{0}}^{2} + \| F_{H}(t) \|_{L^{1}(\mbD)} \right),
	\end{align*}
	which, together with \eqref{STMHD_prel_eq_B_L2l2H0}, implies \eqref{STMHD_prel_eq_B_L2l2H1}.
\end{proof}

\subsection{A Tightness Condition}\label{STMHD_sec_tightness}

This section provides a tightness condition for the path space of the solutions to the tamed MHD equations. As the case of periodic boundary conditions can be treated in the same way, we will focus on the case of the full space $\R^{3}$.

Let $\mbD = \R^{3}$. We endow the space $\mcH_{\text{loc}}^{0}$ of locally $L^{2}$-integrable and divergence-free vector fields with the following Fr\'{e}chet metric: for $y,z \in \mcH_{\text{loc}}^{0}$
\begin{align*}
	\rho(y,z) := \sum_{m \in \N} 2^{-m} \left( \left[ \int_{|x| \leq m} |y(x) - z(x)|^{2} \diff x \right]^{1/2} \wedge 1 \right).
\end{align*}
Then the space $(\mcH_{\text{loc}}^{0}, \rho)$ is a Polish space and $\mcH^{0} \subset \mcH_{\text{loc}}^{0}$.

Now set $\mcX := C(\R_{+}; \mcH_{\text{loc}}^{0})$ and endow it with the metric
\begin{align*}
	\rho_{\mcX}(y,z) := \sum_{m=1}^{\infty} 2^{-m} \left( \sup_{t \in [0,m]} \rho(y(t), z(t)) \wedge 1 \right).
\end{align*}
We fix a complete orthonormal basis $\mcE := \{ e_{k} ~|~ k \in \N \} \subset \mcV$ of $\mcH^{1}$ such that $\spann \{ \mcE \}$ is a dense subset of $\mcH^{3}$ and - in the case $\mbD = \mbT^{3}$ - such that it is also an orthogonal basis of $\mcH^{0}$. Given $y \in \mcH^{0}, z \in \mcH^{2}$, the inner product $\langle y,z \rangle_{\mcH^{1}}$ is understood in the generalised sense:
\begin{equation}\label{STMHD_prel_eq_gen_H1_SP}
	\langle y,z \rangle_{\mcH^{1}} = \langle y, (I-\D)z \rangle_{\mcH^{0}}.
\end{equation}
This works in accordance to our choice of evolution triple
\begin{align*}
	\mcH^{2} \subset \mcH^{1} \subset \mcH^{0}.
\end{align*}
The next lemma provides conditions for a set of $\mcX$ to be relatively compact, similar to the classical theorem of Arzel\'{a} and Ascoli. This compactness condition can then be turned into a tightness condition.
\begin{lemma}[Compactness in $\mcX$]\label{STMHD_thm_tightness_lemma}
Let $K \subset \mcX$ satisfy the following conditions: for every $T > 0$
\begin{enumerate}[label=(\roman*), ref=(\roman*)]
	\item\label{STMHD_prel_itm_unif_bdd} $\sup_{y \in K} \sup_{s \in [0,T]} \| y(s) \|_{\mcH^{1}} < \infty$,
	\item\label{STMHD_prel_itm_equicont} $\lim_{\delta \rightarrow 0} \sup_{y \in K} \sup_{s,t \in [0,T], |t-s| < \delta}| \langle y(t) - y(s), e \rangle_{\mcH^{1}} | = 0$ for all $e \in \mcE$.
\end{enumerate}
Then $K$ is relatively compact in $\mcX$.
\end{lemma}
\begin{proof}
A proof of this statement can be found in \cite[Lemma 2.6, pp. 221 f.]{RZ09b}.
\end{proof}
\newpage
\begin{lemma}[Tightness in $\mcX$]
	Let $(\mu_{n})_{n \in \N}$ be a family of probability measures on $(\mcX, \mcB(\mcX))$ and assume that
	\begin{enumerate}[label=(\roman*), ref=(\roman*)]
		\item\label{STMHD_prel_itm_tightness_ubdd} For all $T > 0$
		\begin{align*}
			\lim_{R \rightarrow \infty} \sup_{n \in \N} \mu_{n} \big\{  y \in \mcX ~|~ \sup_{s \in [0,T]} \| y(s) \|_{\mcH^{1}} > R \big\} = 0.
		\end{align*}
		\item\label{STMHD_prel_itm_tightness_econt} For all $e \in \mcE$ and any $\varepsilon, T >0$
		\begin{align*}
			\lim_{\delta \downarrow 0} \sup_{n \in \N} \mu_{n} \big\{ y \in \mcX ~|~ \sup_{s,t \in [0,T], |s-t| \leq \delta} | \langle y(t) - y(s), e \rangle_{\mcH^{1}} | > \varepsilon \big\} = 0.
		\end{align*}
	\end{enumerate}
	Then $(\mu_{n})_{n \in \N}$ is tight on $(\mcX, \mcB(\mcX))$.
\end{lemma}
\begin{proof}
	This follows in the same way as \cite[Lemma 2.7, pp. 222 f.]{RZ09b}
\end{proof}

\section{Existence and Uniqueness of Strong Solutions}\label{STMHD_sec_exuniq}
In this section, we prove the main results of this paper, namely that there exists a unique strong solution to the TMHD equations. This is done by using a Faedo-Galerkin approximation to get a weak solution and then proving pathwise uniqueness and employing the celebrated Yamada-Watanabe theorem to conclude existence of a unique strong solution.

\subsection{Weak and Strong Solutions}\label{STMHD_ssec_weakstrong}
For a metric space $\mcU$, we denote the set of all probability measures on $\mcU$ by $\mcP(\mcU)$.

\begin{definition}\label{STMHD_ex_def_WS}
	We say Equation \eqref{STMHD_eq_evol_eqn} has a \emph{weak solution} with initial law $\vt \in \mcP(\mcH^{1})$ if there exists a stochastic basis $(\O, \mcF, P, \left( \mcF_{t} \right)_{t \geq 0})$, an $\mcH^{1}$-valued, $(\mcF_{t})_{t}$-adapted process $y$ and two independent infinite sequences of independent standard $(\mcF_{t})$-Brownian motions $\left\{ \mcW^{k}(t) = \begin{pmatrix} W^{k}(t) \\ \bar{W}^{k}(t) \end{pmatrix} ~|~ t \geq 0, k \in \N \right\}$ such that
	\begin{enumerate}[label=(\roman*), ref=(\roman*)]
		\item\label{STMHD_ex_itm_WS_IC} $y(0)$ has law $\vt$ in $\mcH^{1}$;
		\item\label{STMHD_ex_itm_WS_reg} For $P$-a.e. $\o \in \O$ and every $T > 0$, $y(\cdot, \o) \in C([0,T];\mcH^{1}) \cap L^{2}([0,T];\mcH^{2})$;
		\item\label{STMHD_ex_itm_WS_eqn} it holds that in $\mcH^{0}$
		\begin{align*}
			y(t) &= y_{0} + \int_{0}^{t} [\mcA(y(t)) + \mcP f(t,y(t))]\diff t + \sum_{k=1}^{\infty} \int_{0}^{t} \mcB_{k}(t,y(t)) \diff \mcW_{t}^{k},
		\end{align*}
		for all $t \geq 0$, $P$-a.s.
	\end{enumerate}
	The solution is denoted by $(\O, \mcF, P, \left( \mcF_{t} \right)_{t \geq 0}; \mcW; y)$.
\end{definition}
Under the Assumptions \ref{STMHD_prel_itm_ass_H1}--\ref{STMHD_prel_itm_ass_H3} in this section, the above integrals are all well-defined by Equation \eqref{STMHD_prel_eq_A_bdd}.

For weak solutions, there are several notions of uniqueness. In this work, we are concerned mostly with pathwise uniqueness.
\begin{definition}
	We say that \emph{pathwise uniqueness} holds for Equation \eqref{STMHD_eq_evol_eqn} if, whenever we are given two weak solutions on the same stochastic basis with the same Brownian motion,
	\begin{align*}
		(\O, \mcF, P, (\mcF_{t})_{t \geq 0}; \mcW; y), (\O, \mcF, P, (\mcF_{t})_{t \geq 0}; \mcW; y'),
	\end{align*}
	the condition $P\{ y(0) = y'(0) \} = 1$ implies $P\{ y(t) = y'(t), \forall t \geq 0 \} = 1$.
\end{definition}

The following proposition links existence of weak solutions to existence of a solution to a corresponding martingale problem.
\begin{proposition}\label{STMHD_thm_MP}
	Let $\mcE$ and $\mcX$ be given as in Section \ref{STMHD_sec_tightness}. For $\vt \in \mcP(\mcH^{1})$, the following are equivalent:
	\begin{enumerate}[label=(\roman*), ref=(\roman*)]
		\item\label{STMHD_itm_WSchar_WS} There exists a weak solution to Equation \eqref{STMHD_eq_evol_eqn} with initial law $\vt$;
		\item\label{STMHD_itm_WSchar_MP} There exists a measure $P_{\vt} \in \mcP(\mcX)$ with the following property: for $P_{\vt}$-almost all $y \in \mcX$ and any $T > 0$, the real-valued process defined by
		\begin{equation}\label{STMHD_ex_eq_MPregularity}
			y \in L^{\infty}([0,T];\mcH^{1}) \cap L^{2}([0,T];\mcH^{2}),
		\end{equation}
		and for any $\vp \in C_{0}^{\infty}(\R)$ and any $e \in \mcE$
		\begin{align*}
			M_{e}^{\vp}(t,y) &:= \vp(\langle y(t), e \rangle_{\mcH^{1}}) - \vp(\langle y(0), e \rangle_{\mcH^{1}}) - \int_{0}^{t} \vp'(\langle y(s), e \rangle_{\mcH^{1}}) \cdot \langle \mcA(y(s)), e \rangle_{\mcH^{1}} \diff s \\
			&\quad - \int_{0}^{t} \vp'(\langle y(s), e \rangle_{\mcH^{1}}) \cdot \langle f(s,y(s)), e \rangle_{\mcH^{1}} \diff s - \frac{1}{2} \int_{0}^{t} \vp''(\langle y(s), e \rangle_{\mcH^{1}}) \cdot \| \langle \mcB(s,y(s)), e \|_{\ell^{2}} \diff s
		\end{align*}
		is a continuous local $P_{\vt}$-martingale with respect to $(\mcB_{t}(\mcX))_{t}$, where $\mcB_{t}(\mcX)$ denotes the sub $\sigma$-algebra of $\mcX$ up to time $t$.
	\end{enumerate}
\end{proposition}
\begin{proof}
For a proof of this statement, cf. \cite[Section 6.1, pp. 257 ff.]{RZ09b}.
\end{proof}
To be able to define the notion of strong solutions to Equation \eqref{STMHD_eq_evol_eqn}, we need a canonical realisation of an infinite sequence of independent standard Brownian motions on a Polish space. To this end, consider the space $C(\R_{+},\R)$ of continuous functions on $\R_{+}$ with the metric
\begin{align*}
	\tilde{\rho}(w,w') := \sum_{k=1}^{\infty} 2^{-k} \left( \sup_{t \in [0,k]} | w(t) - w'(t)| \wedge 1 \right).
\end{align*}
We define the product space $\mbW := \prod\limits_{j=1}^{\infty} C(\R_{+}; \R)$ and endow it with the metric
\begin{align*}
	\rho_{\mbW}(w,w') := \sum_{j=1}^{\infty} 2^{-j} \left( \tilde{\rho}(w^{j},w'^{j}) \wedge 1 \right), \quad w = (w^{1}, w^{2}, \ldots), \quad w' = (w'^{1}, w'^{2}, \ldots).
\end{align*}
The space $(\mbW, \rho_{\mbW})$ is a Polish space. We denote the $\sigma$-algebra up to time $t$ by $\mcB_{t}(\mbW) \subset \mcB(\mbW)$ and endow $(\mbW, \mcB(\mbW))$ with the Wiener measure $\P$ such that the coordinate process
\begin{align*}
	w(t) := \left( w^{1}(t) , w^{2}(t), \ldots \right)
\end{align*}
is an infinite sequence of independent standard $(\mcB_{t}(\mbW))_{t}$-Brownian motions on the probability space $(\mbW, \mcB(\mbW), \P)$. 
To cater for the two noise terms present in the stochastic MHD equations, we take two copies of $\mbW$ and take their product $\mfW := \mbW \times \mbW$, with the metric
\begin{align*}
	\rho_{\mfW}\left( (w,\bar{w}) , (w', \bar{w}') \right) := \rho_{\mbW}(w,w') + \rho_{\mbW}(\bar{w}, \bar{w}') .
\end{align*}
Then $(\mfW, \rho_{\mfW})$ is also a Polish space. In the same way as above, we introduce the filtration $\mcB_{t}(\mfW) \subset \mcB(\mfW)$, and endow $(\mfW, \mcB(\mfW))$ with the product $\mfP := \P \otimes \P$ of the two Wiener measures. Then the coordinate process
\begin{align*}
	\mcW(t) := \begin{pmatrix} W(t) \\ \bar{W}(t) \end{pmatrix}
\end{align*}
consists of two independent infinite sequences of independent standard Brownian motions. For simplicity, in the following we will sometimes refer to the process $\mcW$ as a Brownian motion.

Now consider the space of continuous, $\mcH^{1}$-valued paths that are square-integrable in time with values in $\mcH^{2}$, $\mfB := C(\R_{+}; \mcH^{1}) \cap L_{\text{loc}}^{2}(\R_{+};\mcH^{2})$ with the metric
\begin{align*}
	\rho_{\mfB}(y,y') := \sum_{k \in \N} 2^{-k} \left( \sup_{t \in [0,k]} \| y(t) - y'(t) \|_{\mcH^{1}} + \int\limits_{0}^{k} \| y(t) - y'(t) \|_{\mcH^{2}}^{2} \diff t \right) \wedge 1.
\end{align*}
The $\sigma$-algebra up to time $t$ of this space is denoted by $\mcB_{t}(\mfB) \subset \mcB(\mfB)$. For any measure space $(S, \mcS, \lambda)$, we denote the completion of the $\sigma$-algebra $\mcS$ with respect to the measure $\lambda$ by $\bar{\mcS}^{\lambda}$.

\begin{definition}\label{STMHD_def_SS}
	Let $(\O, \mcF, P, \left( \mcF_{t} \right)_{t \geq 0}; \mcW; y)$ be a weak solution of Equation \eqref{STMHD_eq_evol_eqn} with initial distribution $\vt \in \mcP(\mcH^{1})$. If there exists a $\overline{\mcB(\mcH^{1}) \times \mcB(\mfW)}^{\vt \otimes P} / \mcB(\mfB)$-measurable function $F_{\vt} \colon \mcH^{1} \times \mfW \rightarrow \mfB$ such that
	\begin{enumerate}[label=(\roman*), ref=(\roman*)]
		\item For every $t > 0$, $F_{\vt}$ is $\hat{\mcB}_{t} / \mcB_{t}(\mfB)$-measurable, where $\hat{\mcB}_{t} := \overline{\mcB(\mcH^{1}) \times \mcB_{t}(\mfW)}^{\vt \otimes P}$,
		\item $y(\cdot) = F_{\vt}(y(0), \mcW(\cdot))$, $P$-a.s.,
	\end{enumerate}
	then we call $(\mcW,y)$ a \emph{strong solution}.
\end{definition}
\begin{remark}
	The function $F_{\vt}$ is a ``machine" that turns an initial value $y_{0} \in \mcH^{1}$ and a Brownian motion $\mcW \in \mfW$ into a solution via $y = F_{\vt}(y_{0},\mcW)$. The first property of $F_{\vt}$ is a type of adaptedness.
\end{remark}
Our next definition serves the purpose to clarify what we mean by a unique solution.
\begin{definition}\label{STMHD_def_USS}
	Equation \eqref{STMHD_eq_evol_eqn} is said to have a \emph{unique strong solution} associated to $\vt \in \mcP(\mcH^{1})$ if there exists a function $F_{\vt} \colon \mcH^{1} \times \mfW \rightarrow \mfB$ as in Definition \ref{STMHD_def_SS} such that also the following two conditions are satisfied:
	\begin{enumerate}[label=(\roman*), ref=(\roman*)]
		\item for any two independent copies of infinite sequence of independent standard Brownian motions $\{ \mcW(t) ~|~ t \geq 0 \}$ on the stochastic basis $(\O, \mcF,P, (\mcF_{t})_{t \geq 0})$ and any $\mcH^{1}$-valued, $\mcF_{0}$-measurable random variable $y_{0}$ with distribution $P \circ y_{0}^{-1} = \vt$,
		\begin{align*}
			(\O, \mcF, P, \left( \mcF_{t} \right)_{t \geq 0}; \mcW; F_{\vt}(y_{0}, \mcW(\cdot))) ~ \text{is~a~weak~solution~of~Equation~\eqref{STMHD_eq_evol_eqn}};
		\end{align*}
		\item for any weak solution $(\O, \mcF, P, \left( \mcF_{t} \right)_{t \geq 0}; \mcW; y)$ of Equation \eqref{STMHD_eq_evol_eqn} with initial law $\vt$,
		\begin{align*}
			y(\cdot) = F_{\vt}(y(0), \mcW(\cdot)), \quad P-a.s.
		\end{align*}
	\end{enumerate}
\end{definition}
The following Yamada-Watanabe theorem shows the relationship between the different definitions introduced in this section and gives a way to show existence of a unique strong solution to our equation.
\begin{theorem}[Yamada-Watanabe]
	If there exists a weak solution to Equation \eqref{STMHD_eq_evol_eqn} and pathwise uniqueness holds, then there exists a unique strong solution to Equation \eqref{STMHD_eq_evol_eqn}.
\end{theorem}
\begin{proof}
	This follows by applying the results of M. R\"ockner, B. Schmuland and X.C. Zhang \cite[Theorem 2.1]{RSZ08}.
\end{proof}

\subsection{Pathwise Uniqueness}\label{STMHD_ssec_uniq}
In this section we prove that pathwise uniqueness holds for the tamed MHD equations, following the ideas of \cite{RZ09b}.
\begin{theorem}[pathwise uniqueness]\label{STMHD_thm_uniq}
	Let the Assumptions \ref{STMHD_prel_itm_ass_H1}--\ref{STMHD_prel_itm_ass_H3} be satisfied. Then pathwise uniqueness holds for \eqref{STMHD_eq_evol_eqn}.
\end{theorem}
\begin{proof}
	Let $y_{1}, y_{2}$ belong to two weak solutions defined on the same stochastic basis $(\O, \mcF, P, \left( \mcF_{t} \right)_{t \geq 0})$ and with the same Brownian motion $\mcW$ and same initial condition $y_{0}$. Let $T > 0$, $R > 0$ and define the stopping time
	\begin{align*}
		\tau_{R} := \inf \{ t \in [0,T] ~|~ \| y_{1}(t) \|_{\mcH^{1}} \vee \| y_{2}(t) \|_{\mcH^{1}} \geq R \} \wedge T.
	\end{align*}
	By Assumption \ref{STMHD_ex_itm_WS_reg} of the definition of a weak solution, we know that for $P$-a.e. $\o$ $y_{i}(\cdot, \o) \in C([0,T];\mcH^{1})$, $i=1,2$, and thus
	\begin{align*}
		\tau_{R} \uparrow T, \quad \mathrm{as~}R \rightarrow \infty, \quad P-a.s.
	\end{align*}
	So $\tau_{R}$ is a localising sequence. Now we set $z(t) := y_{1}(t) - y_{2}(t)$. $z$ satisfies the equation
	\begin{align*}
		\diff z(t) &= \left[ \mcA(y_{1}(t)) - \mcA(y_{2}(t)) + \mcP\left( f(t,y_{1}(t)) - f(t,y_{2}(t)) \right) \right] \diff t \\
		&\quad + \sum_{k=1}^{\infty} \left( \mcB_{k}(t,y_{1}(t)) - \mcB_{k}(t,y_{2}(t)) \diff \mcW_{t}^{k} \right), \\
		z(0) &= 0.
	\end{align*}
	Thus by It\^{o}'s formula for $\| z \|_{\mcH^{0}}^{2}$, we find (noting the self-adjointness of the projection $\mcP$ with respect to $\langle \cdot, \cdot \rangle_{\mcH^{0}} = \langle \cdot, \cdot \rangle_{\mcL^{2}}$)
	\begin{equation}\label{STMHD_ex_eq_pwu_Ito}	
		\begin{split}
	 	\| z(t) \|_{\mcH^{0}}^{2} &= 2 \int_{0}^{t} \langle \mcA(y_{1}(s)) - \mcA(y_{2}(s)), z(s) \rangle_{\mcH^{0}} \diff s + 2 \int_{0}^{t} \langle f(s,y_{1}(s)) - f(s,y_{2}(s)), z(s) \rangle_{\mcH^{0}} \diff s \\
	 	&\quad + 2 \sum_{k=1}^{\infty} \int_{0}^{t} \langle \mcB_{k}(s,y_{1}(s)) - \mcB_{k}(s,y_{2}(s)), z(s) \rangle_{\mcH^{0}} \diff \mcW_{s}^{k} \\
	 	&\quad + \sum_{k=1}^{\infty} \int_{0}^{t} \| \mcB_{k}(s,y_{1}(s)) - \mcB_{k}(s,y_{2}(s)) \|_{\mcH^{0}}^{2} \diff s \\
	 	&=: I_{1}(t) + I_{2}(t) + I_{3}(t) + I_{4}(t).
	 	\end{split} 
	\end{equation}
	We stop with $\tau_{R}$ and estimate each term separately.
	
	The first term $I_{1}(t \wedge \tau_{R})$ can be rewritten and estimated using integration by parts and Young's inequality:
	\begin{align*}
		&I_{1}(t \wedge \tau_{R}) \leq - \int_{0}^{t \wedge \tau_{R}} \| \nabla z(s) \|_{\mcH^{0}}^{2} \diff s \\
		&\quad + 2 \int_{0}^{t \wedge \tau_{R}} \Big[ \| \bm{v}_{1} \otimes \bm{v}_{1} - \bm{v}_{2} \otimes \bm{v}_{2} \|_{\mbL^{2}}^{2}  + \| \bm{B}_{1} \otimes \bm{B}_{1} - \bm{B}_{2} \otimes \bm{B}_{2}  \|_{\mbL^{2}}^{2} \\
		&\quad \quad + \| \bm{v}_{1} \otimes \bm{B}_{1} - \bm{v}_{2} \otimes \bm{B}_{2} \|_{\mbL^{2}}^{2} + \| \bm{B}_{1} \otimes \bm{v}_{1} - \bm{B}_{2} \otimes \bm{v}_{2} \|_{\mbL^{2}}^{2} \Big]\diff s \\
	 &\quad - 2 \int \limits_{0}^{t \wedge \tau_{R}} \left\langle g_{N}(y_{1})y_{1} - g_{N}(y_{2})y_{2} , z\right\rangle_{\mcL^{2}} \diff s =: J_{1}(t \wedge \tau_{R}) + J_{2}(t \wedge \tau_{R}) + J_{3}(t \wedge \tau_{R}).
	\end{align*}
The terms of $J_{2}(t \wedge \tau_{R})$ are of the general form
\begin{align*}
	\| \bm{\theta}_{1} \otimes \bm{\psi}_{1} - \bm{\theta}_{2} \otimes \bm{\psi}_{2} \|_{\mbL^{2}}^{2}, \quad \bm{\theta}_{i}, \bm{\psi}_{i} \in \{ \bm{v}, \bm{B} \}, i = 1,2,
\end{align*}
and can be estimated as follows:
\begin{align*}
	&\| \bm{\theta}_{1} \otimes \bm{\psi}_{1} - \bm{\theta}_{2} \otimes \bm{\psi}_{2} \|_{\mbL^{2}}^{2} \\
	&\leq 2 \left( \| \bm{\theta}_{1} \|_{\mbL^{4}}^{2} \|\bm{\psi}_{1} - \bm{\psi}_{2} \|_{\mbL^{4}}^{2} + \| \bm{\theta}_{1} - \bm{\theta}_{2} \|_{\mbL^{4}}^{2} \| \bm{\psi}_{2} \|_{\mbL^{4}}^{2} \right) \\
	&\leq 2C_{1,4}^{2} \left( \| \bm{\theta}_{1} \|_{\mbH^{1}}^{2} \|\bm{\psi}_{1} - \bm{\psi}_{2} \|_{\mbH^{1}}^{3/2} \|\bm{\psi}_{1} - \bm{\psi}_{2} \|_{\mbH^{0}}^{1/2} + \| \bm{\theta}_{1} - \bm{\theta}_{2} \|_{\mbH^{1}}^{3/2} \| \bm{\theta}_{1} - \bm{\theta}_{2} \|_{\mbH^{0}}^{1/2} \| \bm{\psi}_{2} \|_{\mbH^{1}}^{2} \right) \\
	&\leq 2 C_{1,4}^{2} R^{2} \left( \|\bm{\psi}_{1} - \bm{\psi}_{2} \|_{\mbH^{1}}^{3/2} \|\bm{\psi}_{1} - \bm{\psi}_{2} \|_{\mbH^{0}}^{1/2} + \| \bm{\theta}_{1} - \bm{\theta}_{2} \|_{\mbH^{1}}^{3/2} \| \bm{\theta}_{1} - \bm{\theta}_{2} \|_{\mbH^{0}}^{1/2} \right) \\
	&\leq \frac{1}{16} \|\bm{\psi}_{1} - \bm{\psi}_{2} \|_{\mbH^{1}}^{2} + \frac{1}{16} \| \bm{\theta}_{1} - \bm{\theta}_{2} \|_{\mbH^{1}}^{2} + C_{R}  \left( \|\bm{\psi}_{1} - \bm{\psi}_{2} \|_{\mbH^{0}}^{2} + \|\bm{\theta}_{1} - \bm{\theta}_{2} \|_{\mbH^{0}}^{2} \right).
\end{align*}
Counting all possible combinations of $\bm{\theta}$ and $\bm{\psi}$ and combining the $\bm{v}$ and $\bm{B}$-norms into the corresponding norms for $y$, we find that
\begin{align*}
	J_{2}(t \wedge \tau_{R}) \leq \frac{1}{4} \int_{0}^{t} \| \nabla z(s) \|_{\mcH^{0}}^{2} \diff s + C_{R} \int_{0}^{t} \| z(s) \|_{\mcH^{0}}^{2} \diff s.
\end{align*}
	Since $|g_{N}(r) - g_{N}(r')| \leq 2|r - r'|$, we see, using a short calculation as well as  Sobolev embedding and Young's inequality, that
	\begin{align*}
		J_{3}(t \wedge \tau_{R}) &\leq 8 \int\limits_{0}^{t \wedge \tau_{R}} \left\| |z| \cdot \left( |y_{1}| + |y_{2}| \right) \right\|_{L^{2}}^{2} \diff s \leq 16 \int\limits_{0}^{t \wedge \tau_{R}} \| z \|_{\mcL^{4}}^{2} \left( \| y_{1} \|_{\mcL^{4}}^{2} + \| y_{2} \|_{\mcL^{4}}^{2} \right) \diff s\\
		&\leq 16 C_{1,4}^{2} \int\limits_{0}^{t \wedge \tau_{R}} \| \nabla z \|_{\mcH^{0}}^{3/2} \| z \|_{\mcH^{0}}^{1/2} \left( \| y_{1} \|_{\mcH^{1}}^{2} + \| y_{2} \|_{\mcH^{1}}^{2} \right) \diff s \\
		&\leq 16 C_{1,4}^{2} R^{2} \int\limits_{0}^{t \wedge \tau_{R}} \| \nabla z \|_{\mcH^{0}}^{3/2} \| z \|_{\mcH^{0}}^{1/2} \diff s \\
		&\leq \frac{1}{4} \int\limits_{0}^{t \wedge \tau_{R}} \| \nabla z \|_{\mcH^{0}}^{2} \diff s  + C_{R} \int\limits_{0}^{t \wedge \tau_{R}}  \| z \|_{\mcH^{0}}^{2} \diff s.
	\end{align*}
Thus, altogether we find that 
\begin{align*}
	I_{1}(t \wedge \tau_{R}) \leq - \frac{1}{2} \int\limits_{0}^{t \wedge \tau_{R}} \| \nabla z(s) \|_{\mcH^{0}}^{2} \diff s + C_{R} \int\limits_{0}^{t \wedge \tau_{R}}  \| z(s) \|_{\mcH^{0}}^{2} \diff s. 
\end{align*}
	By Cauchy-Schwarz-Buniakowski and Assumption \ref{STMHD_prel_itm_ass_H1}, we find for $I_{2}(t \wedge \tau_{R})$
	\begin{align*}
		2 \int_{0}^{t} \langle f(s,y_{1}(s)) - f(s,y_{2}(s)), z(s) \rangle_{\mcH^{0}} \diff s \leq C_{T,F} \int_{0}^{t \wedge \tau_{R}} \| z(s) \|_{\mcH^{0}}^{2} \diff s.
	\end{align*}
	The term $I_{3}(t \wedge \tau_{R})$ is a martingale and thus killed upon taking expectations.
	
	For $I_{4}(t \wedge \tau_{R})$ we have by Assumptions \ref{STMHD_prel_itm_ass_H2} and \ref{STMHD_prel_itm_ass_H3}
	\begin{align*}
		&I_{4}(t \wedge \tau_{R}) = \sum_{k=1}^{\infty} \int_{0}^{t \wedge \tau_{R}} \| \mcP \left( (\Sigma^{k} \cdot \nabla) z(s) \right) + \mcP \left( H^{k}(s, y_{1}(s)) - H^{k}(s, y_{2}(s)) \right) \|_{\mcH^{0}}^{2} \diff s \\
		&\leq  \sup_{t \in [0,T], x \in \mbD} \| \Sigma(t,x) \|_{\ell^{2}}^{2} \int_{0}^{t \wedge \tau_{R}} \| \nabla z(s) \|_{\mcH^{0}}^{2} \diff s + C_{T,H} \int_{0}^{t \wedge \tau_{R}} \| z(s) \|_{\mcH^{0}}^{2} \diff s \\
		&\leq \frac{1}{4} \int_{0}^{t \wedge \tau_{R}} \| \nabla z(s) \|_{\mcH^{0}}^{2} \diff s + C_{T,H} \int_{0}^{t \wedge \tau_{R}} \| z(s) \|_{\mcH^{0}}^{2} \diff s.
	\end{align*}
	Hence, if we stop \eqref{STMHD_ex_eq_pwu_Ito} with $\tau_{R}$ and take expectations, we find, using the previous estimates, that
	\begin{align*}
		\E \left[ \| z(t \wedge \tau_{R}) \|_{\mcH^{0}}^{2} \right] \leq C_{R,T} \E \left[ \int\limits_{0}^{t \wedge \tau_{R}}  \| z(s) \|_{\mcH^{0}}^{2} \diff s \right]\leq C_{R,T,f,H} \int\limits_{0}^{t}  \E \left[  \| z(s \wedge \tau_{R}) \|_{\mcH^{0}}^{2} \right]\diff s.
	\end{align*}
	Applying Gronwall's lemma yields that for any $t \in [0,T]$
	\begin{align*}
		\E \left[ \| z(t \wedge \tau_{R}) \|_{\mcH^{0}}^{2} \right] = 0. 
	\end{align*}
	Finally, we employ Fatou's lemma to find
	\begin{align*}
		\E \left[ \| z(t) \|_{\mcH^{0}}^{2} \right] \leq \liminf_{R \rightarrow \infty} \E \left[ \| z(t \wedge \tau_{R}) \|_{\mcH^{0}}^{2} \right] = 0. 
	\end{align*}
	Thus, $z(t) = 0$ for all $t \geq 0$, $P$-a.s., i.e., pathwise uniqueness holds.
\end{proof}
\subsection{Existence of Martingale Solutions}\label{STMHD_ssec_exweak}
The main result of this section is the following:
\begin{theorem}\label{STMHD_thm_ex}
	Under the Assumptions \ref{STMHD_prel_itm_ass_H1}--\ref{STMHD_prel_itm_ass_H3}, for any initial law $\vt \in \mcP(\mcH^{1})$, there exists a weak solution for Equation \eqref{STMHD_eq_evol_eqn} in the sense of Definition \ref{STMHD_ex_def_WS}.
\end{theorem}
The proof proceeds in the usual fashion, analogously to the proof of \cite[Theorem 3.8]{RZ09b} by considering Faedo-Galerkin approximations of our equation and will be carried out at the end of this section. To be more precise, fix a stochastic basis $(\O, \mcF, P, \left( \mcF_{t} \right)_{t \geq 0})$ and two independent infinite sequences of independent standard $\left( \mcF_{t} \right)_{t}$-Brownian motions $\{ \mcW^{k}(t) ~|~ t \geq 0, k \in \N \}$ and an $\mcF_{0}$-measurable random variable $y_{0}$ with initial law $\vt \in \mcP(\mcH^{1})$.

The set $\mcE = \{ e_{i} ~|~ i \in \N \} \subset \mcV$ was chosen as a complete orthonormal basis of $\mcH^{1}$. We consider the finite-dimensional subspaces
\begin{align*}
	\mcH_{n}^{1} := \spann \{ e_{i} ~|~ i = 1,\ldots,n \}
\end{align*}
and consider the projections onto $\mcH_{n}^{1}$, i.e. for $y \in \mcH^{0}$, we define (recalling our convention \eqref{STMHD_prel_eq_gen_H1_SP})
\begin{align*}
	\Pi_{n}y := \sum_{i=1}^{n} \langle y , e_{i} \rangle_{\mcH^{1}} e_{i} = \sum_{i=1}^{n} \langle y , (I - \D)e_{i} \rangle_{\mcH^{0}} e_{i}.
\end{align*}
We want to study the following finite-dimensional stochastic ordinary differential equations in $\mcH_{n}^{1}$ as approximations for our infinite-dimensional equation:
\begin{align*}
	\begin{cases}
		\diff y_{n}(t) &= [\Pi_{n}\mcA(y_{n}(t)) + \Pi_{n}f(t, y_{n}(t))] \diff t + \sum_{k} \Pi_{n}\mcB_{k}(t,y_{n}(t)) \diff \mcW_{t}^{k}, \\
		y_{n}(0) &= \Pi_{n} y_{0}.
	\end{cases}
\end{align*}
Using the Lemmas \ref{STMHD_thm_estAB} and \ref{STMHD_thm_estB}, there exists a constant $C_{n,N}$ such that for any $y \in \mcH_{n}^{1}$ the following growth conditions holds:
\begin{align*}
	\langle y , \Pi_{n}\mcA(y) + \Pi_{n} f(t,y) \rangle_{\mcH^{1}_{n}} &\leq C_{n,N} \left( \| y \|_{\mcH^{1}_{n}}^{2} + 1 \right), \\
	\| \Pi_{n} \mcB(t,y) \|_{L_{2}(\ell^{2} \times \ell^{2}; \mcH^{1}_{n})} &\leq C_{n,N} \left( \| y \|_{\mcH^{1}_{n}}^{2} + 1 \right).
\end{align*}

Therefore, we can employ the theory of stochastic (ordinary) differential equations (cf. \cite[Theorem 3.1.1, p. 56]{LR15}) to find a unique $(\mcF_{t})_{t}$-adapted process $y_{n}(t)$ such that $P$-a.s. for all $t \in [0,T]$ 
\begin{equation}\label{STMHD_ex_SDE}
	\begin{split}
	y_{n}(t) = y_{n}(0) + \int_{0}^{t} \Pi_{n} \mcA(y_{n}(s)) \diff s + \int_{0}^{t} \Pi_{n} f(s,y_{n}(s)) \diff s  + \sum_{k=1}^{\infty} \int_{0}^{t} \Pi_{n} \mcB_{k}(s,y_{n}(s)) \diff \mcW_{s}^{k},
	\end{split}
\end{equation}
and for any $i \leq n$
\begin{align}
	\nonumber \left\langle y_{n}(t), e_{i} \right\rangle_{\mcH^{1}} &= \left\langle y_{n}(0), e_{i} \right\rangle_{\mcH^{1}} + \int\limits_{0}^{t} \langle \mcA(y_{n}(s)), e_{i} \rangle_{\mcH^{1}} \diff s + \int\limits_{0}^{t} \left\langle f(s,y_{n}(s)), e_{i} \right\rangle_{\mcH^{1}} \diff s \\
	\label{STMHD_ex_SDE_component} &\quad + \sum_{k=1}^{\infty} \int\limits_{0}^{t} \left\langle \mcB_{k}(s,y_{n}(s)), e_{i} \right\rangle_{\mcH^{1}} \diff \mcW_{s}^{k}.
\end{align}
Our strategy now is as follows:
\begin{enumerate}[label=\arabic*., ref=\arabic*.]
	\item Prove uniform \emph{a priori} estimates for $y_{n}$.
	\item Use these to prove tightness of the associated laws.
	\item Use Skorokhod's embedding theorem to translate the weak convergence from the previous step to $P$-a.s. convergence of the random variables.
	\item Prove uniform moment estimates for the terms of the associated martingale problem. 
	\item Show convergence in probability of the martingale problems.
\end{enumerate}
We start with the \emph{a priori} estimates.
\begin{lemma}[\emph{a priori} estimates]\label{STMHD_thm_apriori} For any $T > 0$, there exists a constant $C_{T,N,f,H,y_{0}} > 0$ such that for any $n \in \N$
\begin{equation}\label{STMHD_ex_eq_apriori}
	\begin{split}
	\E \left[ \sup_{t \in [0,T]} \| y_{n}(t) \|_{\mcH^{1}}^{2} \right] + \int_{0}^{T} \E \left[ \| y_{n}(s) \|_{\mcH^{2}}^{2} \right] \diff s +  \int_{0}^{T} \E \left[ \| \nabla |y_{n}(s)|^{2} \|_{L^{2}}^{2} \right] \diff s \leq C_{T,N,f,H,y_{0}}.
	\end{split}
\end{equation}
Furthermore, in the periodic case it holds that
\begin{equation}\label{STMHD_ex_eq_apriori_periodic_L4}
	\int_{0}^{t} \E \left[ \| y_{n}(s) \|_{\mcL^{4}}^{4} \right] \diff s \leq C_{T,N,f,H,y_{0}}.
\end{equation}
\end{lemma}
\begin{proof}
    The proof of this proceeds in a standard way: we use It\^{o}'s formula for $\| y_{n}(t) \|_{\mcH^{1}}^{2}$ to find
	\begin{equation}\label{Ito}
		 \begin{split}
            \| y_{n}(t) \|_{\mcH^{1}}^{2} &= \| y_{0} \|_{\mcH^{1}}^{2} + 2 \int_{0}^{t} \langle \mcA(y_{n}(s)), y_{n}(s) \rangle_{\mcH^{1}} \diff s + 2 \int_{0}^{t} \langle f(s,y_{n}(s)), y_{n}(s) \rangle_{\mcH^{1}} \diff s \\
            &\quad + M(t) + \int_{0}^{t} \| \mcB(s,y_{n}(s)) \|_{L_{2}(\ell^{2}; \mcH^{1})}^{2} \diff s
		 \end{split}
	\end{equation}
	where the term $M(t)$ is a continuous martingale and has the representation
	\begin{align*}
		M(t) := 2 \sum_{k=1}^{\infty} \langle \mcB_{k}(s,y_{n}(s)), y_{n}(s) \rangle_{\mcH^{1}} \diff \mcW_{s}^{k}.
	\end{align*}
	Then we use Lemmas \ref{STMHD_thm_estA} and \ref{STMHD_thm_estB} to estimate the terms on the right-hand side, take expectation and apply Young's inequality and Gronwall's lemma to arrive at an inequality of the form 
	\begin{equation}\label{STMHD_ex_eq_apriori_supE}
		\begin{split}		
		&\sup_{t \in [0,T]} \E \| y_{n}(t) \|_{\mcH^{1}}^{2} + \int_{0}^{T} \| y_{n}(s) \|_{\mcH^{2}}^{2} \diff s + \int_{0}^{T} \E \Big[ \| |\bm{v}_{n}| \cdot |\nabla \bm{v}_{n} | \|_{L^{2}}^{2}  \\
		& \quad  + \| |\bm{B}_{n}| \cdot |\nabla \bm{B}_{n} | \|_{L^{2}}^{2} + \| |\bm{v}_{n}| \cdot |\nabla \bm{B}_{n} | \|_{L^{2}}^{2} + \| |\bm{B}_{n}| \cdot |\nabla \bm{v}_{n} | \|_{L^{2}}^{2} \Big] \diff s \\
		&\leq C_{T,N,H,f,y_{0}}.
		\end{split}
	\end{equation}
	To ``exchange" expectation and supremum, we start again from \eqref{Ito}, take suprema first and then expectations, and then use the Burkholder-Davis-Gundy (BDG) inequality (cf. \cite{MR13}, Theorem 1.1) and Young's inequality.  Lemma \ref{STMHD_thm_estB} as well as the estimate \eqref{STMHD_ex_eq_apriori_supE} just obtained yield the claimed estimate.

	In the periodic case, as $\mcE$ is also orthogonal in $\mcH^{0}$, so we get an It\^{o} formula for the $\mcH^{0}$-norms as well and conclude taking expectations and using \eqref{STMHD_prel_eq_A_testing2} of Lemma \ref{STMHD_thm_estA} as well as Equation \eqref{STMHD_ex_eq_apriori}. Gronwall's lemma then implies \eqref{STMHD_ex_eq_apriori_periodic_L4}.
\end{proof}
Next we want to prove tightness of the laws of the solutions to the approximate equations \eqref{STMHD_ex_SDE}. We recall the notation $\mcX := C(\R_{+};\mcH^{0}_{\text{loc}})$ from Section \ref{STMHD_sec_tightness}.
\begin{lemma}[tightness]\label{STMHD_thm_tightness}
	Let $\mu_{n} := P \circ (y_{n})^{-1}$ be the law of $y_{n}$ in $(\mcX, \mcB(\mcX))$. Then the family $(\mu_{n})_{n \in \N}$ is tight on $(\mcX, \mcB(\mcX))$.
\end{lemma}
\begin{proof}
	Let $R > 0$. We set
	\begin{align*}
		\tau_{R}^{n} := \inf \{ t \geq 0 ~|~ \| y_{n}(t) \|_{\mcH^{1}_{n}} \geq R \}.
	\end{align*}
	Then, using the Chebychev inequality as well as the \emph{a priori} estimate \eqref{STMHD_ex_eq_apriori}, we find
	\begin{equation}\label{STMHD_ex_eq_tightness_H1bdd}
		\begin{split}
		\sup_{n \in \N} P(\tau_{R}^{n} < T) &= \sup_{n \in \N} P \left( \sup_{t \in [0,T]} \| y_{n}(t) \|_{\mcH^{1}_{n}} \geq R \right) \\
		&\leq \sup_{n \in \N} \frac{1}{R^{2}} \E \left[ \sup_{t \in [0,T]}  \| y_{n}(t) \|_{\mcH^{1}_{n}}^{2} \right] \leq \frac{C_{T,N,f,H,y_{0}}}{R^{2}}.
		\end{split}
	\end{equation}
	For any $q \geq 2$ and $s,t \in [0,T]$ and any $e \in \mcE$ (whose support we denote by $\mcO$), we find by using Equations \eqref{STMHD_ex_SDE_component}, \eqref{STMHD_prel_eq_A_L3_est} as well as the BDG inequality (cf. \cite{MR13}, Theorem 1.1)
	\begin{align*}
		&\E \left[ \left| \langle y_{n}(t \wedge \tau_{R}^{n}) - y_{n}(s \wedge \tau_{R}^{n}), e \rangle_{\mcH^{1}} \right|^{q} \right] \\
		&\leq 3^{q-1} \Bigg( \E \left[ \left| \int_{s \wedge \tau_{R}^{n}}^{t \wedge \tau_{R}^{n}} \langle \mcA(y_{n}(r)), e \rangle_{\mcH^{1}} \diff r \right|^{q} \right] + \E \left[ \left| \int_{s \wedge \tau_{R}^{n}}^{t \wedge \tau_{R}^{n}} \langle f(r,y_{n}(r)), e \rangle_{\mcH^{1}} \diff r \right|^{q} \right] \\
		&\quad + \E \left[\left| \sum_{k=1}^{\infty} \int_{s \wedge \tau_{R}^{n}}^{t \wedge \tau_{R}^{n}} \langle \mcB_{k}(r,y_{n}(r)), e \rangle_{\mcH^{1}} \diff \mcW_{r}^{k} \right|^{q} \right] \Bigg) \\	
		&\leq C_{e,q} \Bigg( \E \left[ \left| \int_{s \wedge \tau_{R}^{n}}^{t \wedge \tau_{R}^{n}} \left( 1 + \| y_{n} \|_{\mcL^{3}(\mcO)}^{3} \right) \diff r \right|^{q} \right] + \E \left[ \left| \int_{s \wedge \tau_{R}^{n}}^{t \wedge \tau_{R}^{n}} \| f(r,y_{n}(r)) \|_{\mcH^{0}} \diff r \right|^{q} \right] \\
		&\quad + \E \left[\left| \int_{s \wedge \tau_{R}^{n}}^{t \wedge \tau_{R}^{n}} \| \mcB(r,y_{n}(r)) \|_{L_{2}(\ell^{2}\times \ell^{2};\mcH^{0})}^{2} \diff r \right|^{q/2} \right] \Bigg),	
	\end{align*}
	where we transferred all the spatial derivatives onto $e$ via \eqref{STMHD_prel_eq_gen_H1_SP}. Applying the H\"older embedding $\mcL^{6}(\mcO) \subset \mcL^{3}(\mcO)$ as well as the Sobolev embedding $\mcH^{1} \subset \mcL^{6}$, Assumptions \ref{STMHD_prel_itm_ass_H1}, \ref{STMHD_prel_itm_ass_H3} and Equation \eqref{STMHD_prel_eq_B_L2l2H0}, we can see that
	\begin{align*}
		&\E \left[ \left| \langle y_{n}(t \wedge \tau_{R}^{n}) - y_{n}(s \wedge \tau_{R}^{n}), e \rangle_{\mcH^{1}} \right|^{q} \right] \\
		&\leq C_{e,q,T} \Bigg( \E \left[ \left| \int_{s \wedge \tau_{R}^{n}}^{t \wedge \tau_{R}^{n}} \left( 1 + \sup_{r \in [0,T]} \| y_{n}(r) \|_{\mcH^{1}}^{3} \right) \diff r \right|^{q} \right] + \E \left[ \left| \int_{s \wedge \tau_{R}^{n}}^{t \wedge \tau_{R}^{n}} \| y_{n}(r) \|_{\mcH^{1}}^{2} + \| F_{f}(r) \|_{L^{1}(\mbD)} \diff r \right|^{q} \right] \\
		&\quad + \E \left[\left| \int_{s \wedge \tau_{R}^{n}}^{t \wedge \tau_{R}^{n}} \| y_{n}(r) \|_{\mcH^{1}}^{2} + \| F_{H}(r) \|_{L^{1}(\mbD)} \diff r \right|^{q/2} \right] \Bigg) \\
		&\leq C_{e,q,T} \Bigg\{ \E \left[  \left( 1 + \sup_{r \in [0,T \wedge \tau_{R}^{n}]} \| y_{n}(r) \|_{\mcH^{1}}^{3q} \right)  \right] |t-s|^{q}  + \E \left[ \left( \sup_{r \in [0,T\wedge \tau_{R}^{n}]} \| y_{n}(r) \|_{\mcH^{1}}^{2q} + \| F_{f} \|_{L^{\infty}([0,T];L^{1}(\mbD))}^{q} \right) \right]  |t-s|^{q} \\
		&\quad + \E \left[ \left( \sup_{r \in [0,T\wedge \tau_{R}^{n}]} \| y_{n}(r) \|_{\mcH^{1}}^{q} + \| F_{H} \|_{L^{\infty}([0,T];L^{1}(\mbD))}^{q/2} \right) \right]  |t-s|^{q/2}\Bigg\} \\
		&\leq C_{e,q,T,R} \E \left[  \left( 1 + \sup_{r \in [0,T]} \| y_{n}(r) \|_{\mcH^{1}}^{2} + \| F_{f} \|_{L^{\infty}([0,T];L^{1}(\mbD))}^{q} + \| F_{H} \|_{L^{\infty}([0,T];L^{1}(\mbD))}^{q/2} \right)  \right] |t-s|^{q/2} \\
		&\leq C_{e,q,T,R,f,H,N,y_{0}}|t-s|^{q/2},
	\end{align*}
	where we have used in the penultimate step that $q \geq 2$ as well as the definition of $\tau_{R}^{n}$ and hence that $\| y_{n}(r) \|_{\mcH^{1}}^{q} = \| y_{n}(r) \|_{\mcH^{1}}^{q-2} \cdot \| y_{n}(r) \|_{\mcH^{1}}^{2} \leq R^{q-2} \| y_{n}(r) \|_{\mcH^{1}}^{2}$ and similarly for the other terms. Finally, we have used the \emph{a priori} estimate \eqref{STMHD_ex_eq_apriori} in the last step. 
	
	Thus by the Kolmogorov-\v{C}entsov continuity criterion (e.g. in the version of \cite[Theorem 3.1, p. 28 f.]{FH14}) with parameters $\beta = \frac{q/2}{q} = 1/2$, we find that for every $T > 0$ and $\alpha \in (0, \frac{1}{2} - \frac{1}{q})$
	\begin{align*}
		\E \left[ \sup_{s,t \in [0,T], |t-s| \leq \delta} | \langle y_{n}(t \wedge \tau_{R}^{n}) - y_{n}(s \wedge \tau_{R}^{n}), e \rangle_{\mcH^{1}} |^{q} \right] \leq C_{e,q,T,R,f,H,N,y_{0}} \cdot \delta^{\alpha}.
	\end{align*}
	Therefore, for arbitrary $\ve > 0$ and $R> 0$, we find using \eqref{STMHD_ex_eq_tightness_H1bdd}
	\begin{align*}
		&\sup_{n \in \N} P \left\{ \sup_{s,t \in [0,T], |t-s| \leq \delta} | \langle y_{n}(t) - y_{n}(s), e \rangle_{\mcH^{1}} | > \ve \right\} \\
		&\leq \sup_{n \in \N} P \left\{ \sup_{s,t \in [0,T], |t-s| \leq \delta} | \langle y_{n}(t) - y_{n}(s), e \rangle_{\mcH^{1}} | > \ve; \tau_{R}^{n} \geq T \right\} + \sup_{n \in \N} P \left\{  \tau_{R}^{n} < T \right\} \\
		&\leq  \frac{C_{e,q,T,R,f,H,N,y_{0}} \cdot \delta^{\alpha}}{\varepsilon^{q}} + \frac{C_{T,N}}{R^{2}}.
	\end{align*}
	Letting first $\delta \downarrow 0$ and then $R \rightarrow \infty$, we find
	\begin{equation}\label{STMHD_ex_eq_tightness_equicont}
		\lim_{\delta \downarrow 0} \sup_{n \in \N} P \left\{ \sup_{s,t \in [0,T], |t-s| \leq \delta} | \langle y_{n}(t) - y_{n}(s), e \rangle_{\mcH^{1}} | > \ve \right\} = 0.
	\end{equation}
	Thanks to \eqref{STMHD_ex_eq_tightness_H1bdd} and \eqref{STMHD_ex_eq_tightness_equicont}, we can now invoke Lemma \ref{STMHD_thm_tightness_lemma} to conclude that $(\mu_{n})_{n \in \N}$ is a tight family of probability measures on $(\mcX, \mcB(\mcX))$.
\end{proof}
	The tightness implies the existence of a subsequence (which we again denote by $(\mu_{n})_{n \in \N}$) that converges weakly to a measure $\mu \in \mcP(\mcX)$.
	
	Next we apply Skorokhod's coupling theorem (cf. \cite[Theorem 4.30, p. 79]{Kallenberg02}) to infer the existence of a probability space $(\tilde{\O}, \tilde{\mcF},\tilde{P})$ and $\mcX$-valued random variables $\tilde{y}^{n}$ and $\tilde{y}$ such that
	\begin{enumerate}[label=(\Roman*), ref=(\Roman*)]
		\item\label{STMHD_ex_itm_Skorokhod_law} the law of $\tilde{y}^{n}$ is the same as that of $y_{n}$ for all $n \in \N$, i.e. $\tilde{P} \circ \tilde{y}_{n}^{-1} = \mu_{n}$;
		\item\label{STMHD_ex_itm_Skorokhod_convPas} the convergence $\tilde{y}_{n} \rightarrow \tilde{y}$ holds in $\mcX$, $\tilde{P}$-a.s., and $\tilde{y}$ has law $\mu$.
	\end{enumerate}
	By Fatou's lemma and the uniform (in $n$) \emph{a priori} estimates \eqref{STMHD_ex_eq_apriori}, the same estimates also hold for the limiting process: for every $T > 0$
	\begin{equation}\label{STMHD_eq_Skorohod_H1est}
			\E^{\tilde{P}} \left[ \sup_{t \in [0,T]} \| \tilde{y}(t) \|_{\mcH^{1}}^{2} \right] \leq \liminf_{n \rightarrow \infty} \E^{\tilde{P}} \left[ \sup_{t \in [0,T]} \| \tilde{y}_{n}(t) \|_{\mcH^{1}}^{2} \right] < \infty,
	\end{equation}
as well as
		\begin{equation}
			\begin{split}
			\int_{0}^{T} \E^{\tilde{P}} \left[  \| \tilde{y}(s) \|_{\mcH^{2}}^{2} \right] \diff s < \infty,
			\end{split}
		\end{equation}
and
		\begin{equation}\label{STMHD_eq_Skorohod_NLest}
			\int\limits_{0}^{T} \E^{\tilde{P}} \Big[ \| |\tilde{\bm{v}}| \cdot |\nabla \tilde{\bm{v}} | \|_{L^{2}}^{2}   + \| |\tilde{\bm{B}}| \cdot |\nabla \tilde{\bm{B}} | \|_{L^{2}}^{2} + \| |\tilde{\bm{v}}| \cdot |\nabla \tilde{\bm{B}} | \|_{L^{2}}^{2}  + \| |\tilde{\bm{B}}| \cdot  |\nabla \tilde{\bm{v}} | \|_{L^{2}}^{2} \Big] \diff s 
			< \infty.
		\end{equation}
In the periodic case $\mbD = \mbT^{3}$, we additionally have
		\begin{equation}\label{STMHD_eq_Skorohod_periodic}
				\int_{0}^{t} \E^{\tilde{P}} \left[ \| \tilde{y}(s) \|_{\mcL^{4}}^{4} \right] \diff s \leq \liminf_{n \rightarrow \infty} \int_{0}^{t} \E \left[ \| y_{n}(s) \|_{\mcL^{4}}^{4} \right] \diff s < \infty.
		\end{equation}
		Next we want to study the martingale problem associated to our solutions $\tilde{y}_{n}$ and prove that their limit solves a martingale problem as well, giving existence of a weak solution to the TMHD equations. To this end, take any $\vp \in C_{c}^{\infty}(\R)$, $e \in \mcE$, $t \geq 0$, $y \in \mcX$. We define the following process:
		\begin{align*}
			M_{e}^{\vp}(t,y) := I_{1}^{\vp}(t,y) -  I_{2}^{\vp}(t,y) -  I_{3}^{\vp}(t,y) -  I_{4}^{\vp}(t,y) -  I_{5}^{\vp}(t,y),
		\end{align*}
		where
		\begin{align*}
			I_{1}^{\vp}(t,y) &:= \vp(\langle y(t) , e \rangle_{\mcH^{1}}), \\
			I_{2}^{\vp}(t,y) &:= \vp(\langle y(0) , e \rangle_{\mcH^{1}}), \\
			I_{3}^{\vp}(t,y) &:= \int_{0}^{t} \vp'(\langle y(0) , e \rangle_{\mcH^{1}}) \cdot \langle \mcA(y(s)),e \rangle_{\mcH^{1}} \diff s, \\
			I_{4}^{\vp}(t,y) &:= \int_{0}^{t} \vp'(\langle y(0) , e \rangle_{\mcH^{1}}) \cdot \langle f(s,y(s)) , e \rangle_{\mcH^{1}} \diff s, \\
			I_{5}^{\vp}(t,y) &:= \frac{1}{2} \int_{0}^{t} \vp''(\langle y(0) , e \rangle_{\mcH^{1}}) \cdot \| \langle \mcB(s,y(s)) , e \rangle_{\mcH^{1}} \|_{\ell^{2}}^{2} \diff s.
		\end{align*}
Our aim now is to show that $M_{e}^{\vp}(\cdot,y)$ is a martingale with respect to the stochastic basis $(\mcX, \mcB(\mcX), \mu$, $(\mcB_{t}(\mcX))_{t \geq 0})$, i.e. it solves the martingale problem. This implies the existence of a weak solution.

Since $e \in \mcE \subset \mcV$, as noted before, it has compact support, i.e. there exists an $m \in \N$ such that $\supp(e) \subset \mcO := \overline{B_{m}(0)} \subset \R^{3}$.

In order to show convergence of the martingale problems, we need to prove uniform moment estimates as well as convergence in probability. The next lemma provides the moment estimates.
\begin{lemma}[uniform integrability of $M_{e}^{\vp}$]
	The following estimate holds
	\begin{equation}\label{STMHD_ex_eq_unif_int}
		\sup_{n \in \N} \E^{\tilde{P}} \left[ \left| M_{e}^{\vp}(t,\tilde{y}_{n}) \right|^{4/3} \right] + \E^{\tilde{P}} \left[ \left| M_{e}^{\vp}(t,\tilde{y}) \right|^{4/3} \right] < \infty.
	\end{equation}
\end{lemma}
\begin{proof}
	We show that each of the terms of $M_{e}^{\vp}$ are bounded. The terms $I_{1,2}^{\vp}$ are obviously bounded by a constant $C_{\vp}$, since $\vp \in C_{c}^{\infty}$.
	
	For $I_{3}^{\vp}$ we have by Jensen's inequality for the temporal integral, as well as Equations \eqref{STMHD_prel_eq_A_L3_est} and \eqref{STMHD_eq_GNS_aux_est}
	\begin{align*}
		\E^{\tilde{P}} \left[ |I_{3}^{\vp}(t, \tilde{y}_{n}) |^{4/3} \right] &\leq T^{4/3-1} \| \vp' \|_{L^{\infty}}^{4/3} \int_{0}^{T} \E^{\tilde{P}} \left[ |\langle \mcA(\tilde{y}_{n}(s)), e \rangle_{\mcH^{1}}|^{4/3} \right] \diff s \\
		&\leq C_{T,\vp,e} \int_{0}^{T} \E^{\tilde{P}} \left[ 1 + \| \tilde{y}_{n}(s) \|_{\mcL^{3}(\mcO)}^{4} \right] \diff s \\
		&\leq C_{T,\vp,e} \int_{0}^{T} \E^{\tilde{P}} \left[ 1 + \| \tilde{y}_{n}(s) \|_{\mcL^{12}(\mcO)}^{4} \right] \diff s \\
		&\leq C_{T,\vp,e} \int_{0}^{T} \E^{\tilde{P}} \left[ 1 + \left\| |\tilde{\bm{v}}| |\nabla \tilde{\bm{v}}| \right\|_{L^{2}}^{2} + \left\| |\tilde{\bm{B}}| |\nabla \tilde{\bm{B}}| \right\|_{L^{2}}^{2}\right] \diff s.
	\end{align*}
	This last term is bounded by our \emph{a priori} estimates \eqref{STMHD_eq_Skorohod_NLest}.

	In the case of periodic domain $\mbD = \mbT^{3}$, we have by Equations \eqref{STMHD_prel_eq_A_L3_est} and \eqref{STMHD_eq_Skorohod_periodic}
	\begin{align*}
		\E^{\tilde{P}} \left[ |I_{3}^{\vp}(t, \tilde{y}_{n}) |^{4/3} \right] &\leq T^{4/3-1} \| \vp \|_{L^{\infty}}^{4/3} \int_{0}^{T} \E^{\tilde{P}} \left[ |\langle \mcA(\tilde{y}_{n}(s)), e \rangle_{\mcH^{1}}|^{4/3} \right] \diff s \\
		&\leq C_{T,\vp,e} \int_{0}^{T} \E^{\tilde{P}} \left[ 1 + \| \tilde{y}_{n}(s) \|_{\mcL^{3}(\mbT^{3})}^{4} \right] \diff s \\
		&\leq C_{T,\vp,e} \int_{0}^{T} \E^{\tilde{P}} \left[ 1 + \| \tilde{y}_{n}(s) \|_{\mcL^{4}(\mbT^{3})}^{4} \right] \diff s \\
		&\leq C_{T,\vp,e,N,f,H,y_{0}}.
	\end{align*}
	The other two terms can be dealt with swiftly: by Jensen's inequality for the convex function $x \mapsto x^{4/3}$, H\"older's inequality for $p = 3/2, q = 3$, Jensen's inequality for the concave function $x \mapsto x^{2/3}$, Assumption \ref{STMHD_prel_itm_ass_H1} and \eqref{STMHD_eq_Skorohod_H1est} we find
	\begin{align*}
		\E^{\tilde{P}} \left[ |I_{4}^{\vp}(t, \tilde{y}_{n}) |^{4/3} \right] &\leq C_{T,\vp} \|e \|_{\mcH^{2}} \int_{0}^{T} \E^{\tilde{P}} \left[ 1 \cdot \| f(s,\tilde{y}_{n}(s)) \|_{\mcH^{0}}^{4/3} \right] \diff s \\
		&\leq C_{T,\vp,e} \int_{0}^{T} \left(\E^{\tilde{P}} \left[ \| f(s,\tilde{y}_{n}(s)) \|_{\mcH^{0}}^{2} \right] \right)^{2/3} \diff s \\
		&\leq C_{T,\vp,e} \left( \int_{0}^{T} \left( 1 + \E^{\tilde{P}} \left[ \| \tilde{y}_{n}(s) \|_{\mcH^{0}}^{2} \right] + \| F_{f}(s) \|_{L^{1}(\mbD)} \right) \diff s \right)^{2/3} \\
		&\leq C_{T,\vp,e} \left(1 + \E^{\tilde{P}}\left[ \sup_{s \in [0,T]} \| \tilde{y}_{n}(s) \|_{\mcH^{1}}^{2} \right] + \| F_{f} \|_{L^{1}([0,T] \times \mbD)} \right)^{2/3}  \\
		&\leq C_{T,\vp, e, N, f,H, y_{0}},
	\end{align*}
	and by Equation \eqref{STMHD_prel_eq_B_H1_est}, we have in a similar manner
	\begin{align*}
		\E^{\tilde{P}} \left[ |I_{5}^{\vp}(t, \tilde{y}_{n}) |^{4/3} \right] &\leq C_{T,\vp} \int_{0}^{T} \E^{\tilde{P}} \left[ \| \mcB(s,\tilde{y}_{n}(s)) \|_{L_{2}(\ell^{2}\times \ell^{2};\mcH^{0})}^{2 \cdot 4/3} \right] \diff s \\
		&\leq C_{T,\vp,e,H,\Sigma} \int_{0}^{T} 1 + \E^{\tilde{P}} \left[ \| \tilde{y}_{n}(s) \|_{\mcL^{2}(\mcO)}^{8/3} \right] \diff s \\
		&\leq C_{T,\vp,e,H,\Sigma} \int_{0}^{T} 1 + \E^{\tilde{P}} \left[ \| \tilde{y}_{n}(s) \|_{\mcL^{3}(\mcO)}^{4} \right] \diff s \\
		&\leq C_{T,\vp,e,H,\Sigma} \int_{0}^{T} 1 + \E^{\tilde{P}} \left[ \| \tilde{y}_{n}(s) \|_{\mcL^{12}(\mcO)}^{4} \right] \diff s \\
		&\leq C_{T,\vp,e,H,\Sigma} \int_{0}^{T} \E^{\tilde{P}} \left[ 1 + \left\| |\tilde{\bm{v}}| |\nabla \tilde{\bm{v}}| \right\|_{L^{2}}^{2} + \left\| |\tilde{\bm{B}}| |\nabla \tilde{\bm{B}}| \right\|_{L^{2}}^{2}\right] \diff s \\
		&\leq C_{T,\vp,e,H,\Sigma,N,y_{0}}.
	\end{align*}
\end{proof}
Our next and last ingredient for the existence proof of weak solutions is the convergence in probability of the martingales defined above, as in \cite[Lemma 3.12, p. 236 f.]{RZ09b}.
\begin{lemma}[convergence in probability]
	For every $t > 0$, $M_{e}^{\vp}(t,\tilde{y}_{n}) \rightarrow M_{e}^{\vp}(t,\tilde{y})$ in probability, i.e. for every $\varepsilon > 0$
	\begin{equation}\label{STMHD_ex_eq_Pconv}
		\lim_{n \rightarrow \infty} \tilde{P} \left\{  \left| M_{e}^{\vp}(t,\tilde{y}_{n}) - M_{e}^{\vp}(t,\tilde{y}) \right| > \varepsilon \right\} = 0.
	\end{equation}
\end{lemma}
\begin{proof}
	As before, we show convergence of all the terms of $M_{e}^{\vp}$ and we denote $\mcO := \supp(e)$.
	Note that by \ref{STMHD_ex_itm_Skorokhod_convPas}, and since $\mcO$ is bounded, we find by the definition of the metric on $\mcX$ that
	\begin{equation}
		\lim_{n \rightarrow \infty} \int_{\mcO} | \tilde{y}_{n}(t,x,\tilde{\o}) - \tilde{y}(t,x,\tilde{\o}) |^{2} \diff x = 0, \quad \tilde{P}-\mathrm{a.a.~} \tilde{\o} \in \tilde{\O}.
	\end{equation}
	This then implies that for $P$-a.a. $\tilde{\o} \in \tilde{\O}$ and all $t \in [0,T]$
	$$
		| \langle \tilde{y}_{n}(t,\cdot,\tilde{\o}) - \tilde{y}(t,\cdot,\tilde{\o}), e \rangle_{\mcH^{1}} | \leq  \| \tilde{y}_{n}(t,\cdot,\tilde{\o}) - \tilde{y}(t,\cdot,\tilde{\o}) \|_{\mcL^{2}(\mcO)} \| e \|_{\mcH^{2}} \rightarrow 0,
	$$
	as $n \rightarrow \infty$. By the continuity and boundedness of $\vp$ and Lebesgue's dominated convergence theorem, we find
	\begin{align*}
		\lim_{n \rightarrow \infty} \E^{\tilde{P}} \left[|I_{1}^{\vp}(t,\tilde{y}_{n}) - I_{1}^{\vp}(t,\tilde{y})| \right] = \lim_{n \rightarrow \infty} \E^{\tilde{P}} \left[ \left| \vp \left(\langle \tilde{y}_{n}(t), e \rangle_{\mcH^{1}} \right) - \vp \left( \langle \tilde{y}(t), e \rangle_{\mcH^{1}} \right) \right| \right] = 0.
	\end{align*}
	Similarly one can show
	\begin{align*}
		\lim_{n \rightarrow \infty} \E^{\tilde{P}} \left[|I_{2}^{\vp}(t,\tilde{y}_{n}) - I_{2}^{\vp}(t,\tilde{y})| \right] = \lim_{n \rightarrow \infty} \E^{\tilde{P}} \left[ \left| \vp \left(\langle \tilde{y}_{n}(0), e \rangle_{\mcH^{1}} \right) - \vp \left( \langle \tilde{y}(0), e \rangle_{\mcH^{1}} \right) \right| \right] = 0.
	\end{align*}
	For $I_{3}^{\vp}$, we define for any $R > 0$ and $n \in \N$ the stopping times
	\begin{align*}
		\tilde{\tau}_{R}^{n} := \inf \{t \geq 0 ~|~ \| \tilde{y}_{n}(t) \|_{\mcH^{1}} \geq R \}.
	\end{align*}
	Then by Chebychev's inequality, \ref{STMHD_ex_itm_Skorokhod_law} and \eqref{STMHD_ex_eq_apriori}, we get
	\begin{equation}\label{STMHD_ex_eq_convP_STest}
		\begin{split}
		\sup_{n \in \N} \tilde{P} \left\{ \tilde{\tau}_{R}^{n} \leq T \right\} &= \sup_{n \in \N} \tilde{P} \left\{ \| \tilde{y}_{n}(t) \|_{\mcH^{1}} \geq R \right\} \leq \sup_{n \in \N} \frac{1}{R^{2}}\E^{\tilde{P}} \left[ \| \tilde{y}_{n}(t) \|_{\mcH^{1}}^{2} \right] \\
		&\leq  \sup_{n \in \N} \frac{1}{R^{2}}\E^{P} \left[ \| {y}_{n}(t) \|_{\mcH^{1}}^{2} \right] \leq \frac{C_{T,N,f,H,y_{0}}}{R^{2}}.
		\end{split}
	\end{equation}
	For arbitrary $R > 0$ we thus find
	\begin{align*}
		&\lim_{n \rightarrow \infty} \tilde{P} \left\{|I_{3}^{\vp}(t,\tilde{y}_{n}) - I_{3}^{\vp}(t,\tilde{y})| > \varepsilon \right\} \\
		&\leq \lim_{n \rightarrow \infty} \tilde{P} \left\{|I_{3}^{\vp}(t,\tilde{y}_{n}) - I_{3}^{\vp}(t,\tilde{y})| > \varepsilon , \tilde{\tau}_{R}^{n} > T \right\} + \frac{C_{T,N,f,H,y_{0}}}{R^{2}}.
	\end{align*}
Equations \eqref{STMHD_prel_eq_A_L3_est}, \eqref{STMHD_prel_eq_stability_est} of Lemma \ref{STMHD_thm_estAB} -- and again the continuity and boundedness of $\vp$ -- imply that for any $t \in [0,T]$ and all $\tilde{\o} \in \tilde{\O}$ with $\tilde{\tau}_{R}^{n}(\tilde{\o}) > T$
	\begin{align*}
		&\left| \vp'(\langle \tilde{y}_{n}(t,\tilde{\o}), e \rangle_{\mcH^{1}}) \langle \mcA(\tilde{y}_{n}(t,\tilde{\o})), e \rangle_{\mcH^{1}} - \vp'(\langle \tilde{y}(t,\tilde{\o}), e \rangle_{\mcH^{1}}) \langle \mcA(\tilde{y}(t,\tilde{\o})), e \rangle_{\mcH^{1}} \right| \\
		&\leq \left| \vp'(\langle \tilde{y}_{n}(t,\tilde{\o}), e \rangle_{\mcH^{1}}) - \vp'(\langle \tilde{y}(t,\tilde{\o}), e \rangle_{\mcH^{1}}) \right| \left| \langle \mcA(\tilde{y}_{n}(t,\tilde{\o})), e \rangle_{\mcH^{1}} \right| \\
		&\quad + \left| \vp'(\langle \tilde{y}(t,\tilde{\o}), e \rangle_{\mcH^{1}}) \right| \left| \langle \mcA(\tilde{y}(t,\tilde{\o})) - \mcA(\tilde{y}_{n}(t,\tilde{\o})), e \rangle_{\mcH^{1}} \right| \rightarrow 0, \quad \mathrm{as~} n \rightarrow \infty.
	\end{align*}
	This -- combined with Markov's inequality and Lebesgue's dominated convergence theorem -- yields the desired convergence:
	\begin{align*}
		&\lim_{n \rightarrow \infty} \tilde{P} \left\{|I_{3}^{\vp}(t,\tilde{y}_{n}) - I_{3}^{\vp}(t,\tilde{y})| > \varepsilon \right\} \\
		&\leq \lim_{R \rightarrow \infty} \lim_{n \rightarrow \infty} \tilde{P} \left\{|I_{3}^{\vp}(t,\tilde{y}_{n}) - I_{3}^{\vp}(t,\tilde{y})| > \varepsilon , \tilde{\tau}_{R}^{n} > T \right\}	\\
		&\leq \lim_{R \rightarrow \infty} \lim_{n \rightarrow \infty} \frac{1}{\varepsilon} \E^{\tilde{P}} \left[ 1_{ \{\tilde{\tau}_{R}^{n} > T \}} |I_{3}^{\vp}(t,\tilde{y}_{n}) - I_{3}^{\vp}(t,\tilde{y})| \right] \\
		&\leq \lim_{R \rightarrow \infty}  \frac{1}{\varepsilon} \E^{\tilde{P}} \bigg[ \lim_{n \rightarrow \infty} \int_{0}^{t} 1_{ \{\tilde{\tau}_{R}^{n} > T \}} \cdot \Big| \vp'(\langle \tilde{y}_{n}(t), e \rangle_{\mcH^{1}}) \langle \mcA(\tilde{y}_{n}(t)), e \rangle_{\mcH^{1}} \\
		&\quad\quad\quad\quad- \vp'(\langle \tilde{y}(t), e \rangle_{\mcH^{1}}) \langle \mcA(\tilde{y}(t)), e \rangle_{\mcH^{1}} \Big| \diff s \bigg] = 0.
	\end{align*}
	Similarly, for $I_{4}^{\vp}$ we get by Assumption \ref{STMHD_prel_itm_ass_H1} and similar convergence arguments that
	$$
		\lim_{n \rightarrow \infty} \tilde{P} \left\{|I_{4}^{\vp}(t,\tilde{y}_{n}) - I_{4}^{\vp}(t,\tilde{y})| > \varepsilon \right\} = 0.
	$$	
Finally, for $I_{5}^{\vp}$ we find that since
\begin{align*}
	&\left| \| \langle \mcB(s,\tilde{y}_{n}(s)) , e \rangle_{\mcH^{1}} \|_{\ell^{2}}^{2} - \| \langle \mcB(s,\tilde{y}(s)) , e \rangle_{\mcH^{1}} \|_{\ell^{2}}^{2} \right| \\
	&\leq \left| \sum_{k=1}^{\infty} | \langle \mcB_{k}(s,\tilde{y}_{n}(s)) , e \rangle_{\mcH^{1}} |^{2} - | \langle \mcB_{k}(s,\tilde{y}(s)) , e \rangle_{\mcH^{1}} |^{2} \right| \\
	&\leq \left| \sum_{k=1}^{\infty} \left| \langle \mcB_{k}(s,\tilde{y}_{n}(s)) - \mcB_{k}(s,\tilde{y}(s)) , e \rangle_{\mcH^{1}} \right| \cdot  \left| \langle \mcB_{k}(s,\tilde{y}_{n}(s)) , e \rangle_{\mcH^{1}}  +  \langle \mcB_{k}(s,\tilde{y}(s)) , e \rangle_{\mcH^{1}} \right|  \right| \\
	&\leq \| e \|_{\mcH^{2}} \left| \sum_{k=1}^{\infty} \left| \langle \mcB_{k}(s,\tilde{y}_{n}(s)) - \mcB_{k}(s,\tilde{y}(s)) , e \rangle_{\mcH^{1}} \right| \cdot  \left( \| \mcB_{k}(s,\tilde{y}_{n}(s))\|_{\mcH^{0}}  +  \| \mcB_{k}(s,\tilde{y}(s)) \|_{\mcH^{0}} \right)  \right| \\
	&\leq C_{e}  \left\| \langle \mcB(s,\tilde{y}_{n}(s)) - \mcB(s,\tilde{y}(s)) , e \rangle_{\mcH^{1}} \right\|_{\ell^{2}} \cdot  \left( \| \mcB(s,\tilde{y}_{n}(s))\|_{L_{2}(\ell^{2}\times \ell^{2};\mcH^{0})}  +  \| \mcB(s,\tilde{y}(s)) \|_{L_{2}(\ell^{2}\times \ell^{2};\mcH^{0})} \right),
\end{align*}
when combining this with \eqref{STMHD_prel_eq_B_L2l2H0}, the convergence 
$$
	\lim_{n \rightarrow \infty} \tilde{P} \left\{|I_{5}^{\vp}(t,\tilde{y}_{n}) - I_{5}^{\vp}(t,\tilde{y})| > \varepsilon \right\} = 0,
$$
follows. This concludes the proof.
\end{proof}
We are now in a position to prove the main result of this section.
\begin{proof}[Proof of Theorem \ref{STMHD_thm_ex}]
	Let $t > s$, $G$ be a bounded $\R$-valued, $\mcB_{s}(\mcX)$-measurable continuous function on $\mcX$. Then we have by the generalised Lebesgue theorem
	\begin{alignat*}{3}
		&\E^{\mu} \left[ \left( M_{e}^{\vp}(t,y) - M_{e}^{\vp}(s,y) \right) \cdot G(y) \right] & & \\
		&= \E^{\tilde{P}} \left[ \left( M_{e}^{\vp}(t,\tilde{y}) - M_{e}^{\vp}(s,\tilde{y}) \right) \cdot G(\tilde{y}) \right] & &(\mu = \tilde{P}\circ \tilde{y}^{-1} )\\
		&= \lim_{n \rightarrow \infty} \E^{\tilde{P}} \left[ \left( M_{e}^{\vp}(t,\tilde{y}_{n}) - M_{e}^{\vp}(s,\tilde{y}_{n}) \right) \cdot G(\tilde{y}_{n}) \right] & & \left( \eqref{STMHD_ex_eq_unif_int}, \eqref{STMHD_ex_eq_Pconv} \right)\\
		&= \lim_{n \rightarrow \infty} \E^{P} \left[ \left( M_{e}^{\vp}(t,y_{n}) - M_{e}^{\vp}(s,y_{n}) \right) \cdot G(y_{n}) \right] \quad\quad & &  (\tilde{P} \circ \tilde{y}_{n}^{-1} = P \circ y_{n}^{-1}) \\
		&= 0. & &
	\end{alignat*}
	The last step uses the martingale property of $M_{e}^{\vp}(s,y_{n})$ on the stochastic basis $(\O,\mcF,P,(\mcF_{t})_{t \geq 0})$ and since $G(y_{n})$ is $\mcF_{s}$-measurable. Therefore, $\mu$ solves the martingale problem and by Proposition \ref{STMHD_thm_MP} we infer the existence of a weak solution to \eqref{STMHD_eq_evol_eqn}.
\end{proof}
\subsection{Proof of Theorem \ref{STMHD_thm_exuniq}}
\begin{proof}[Proof of Theorem \ref{STMHD_thm_exuniq}]
	The only thing left to prove is estimate \eqref{STMHD_ex_eq_aprioriNbound}. By It\^{o}'s formula, taking expectations and using \eqref{STMHD_prel_eq_A_testing1}, Assumption \ref{STMHD_prel_itm_ass_H1} and Equation \eqref{STMHD_prel_eq_B_L2l2H0}, we find
	\begin{align*}
		\E \| y(t) \|_{\mcH^{0}}^{2} &= \| y_{0} \|_{\mcH^{0}}^{2} + 2 \int_{0}^{t} \E \langle \mcA(y(s)), y(s) \rangle_{\mcH^{0}} \diff s \\
		&\quad + 2 \int_{0}^{t} \E \langle f(s,y(s)), y(s) \rangle_{\mcH^{0}} \diff s + \int_{0}^{t} \E \| \mcB(s,y(s)) \|_{L_{2}(\ell^{2}\times \ell^{2};\mcH^{0})}^{2} \diff s \\
		&\leq \| y_{0} \|_{\mcH^{0}}^{2} - \frac{3}{2} \int\limits_{0}^{t} \E \| y(s) \|_{\mcH^{1}}^{2} \diff s + C_{f,H} + C_{T,f,H} \int\limits_{0}^{t} \E \| y(s) \|_{\mcH^{0}}^{2} \diff s,
	\end{align*}
	which by Gronwall's lemma implies 
	\begin{align*}
		\E \| y(t) \|_{\mcH^{0}}^{2} + \int\limits_{0}^{t} \E \| y(s) \|_{\mcH^{1}}^{2} \diff s \leq C_{T,f,H} \left( \| y_{0} \|_{\mcH^{0}}^{2} + 1 \right).
	\end{align*}
	Applying Burkholder's inequality and performing similar calculations as in the proof of Lemma \ref{STMHD_thm_apriori}, we get the desired estimate.
\end{proof}
\section{Feller Property and Existence of Invariant Measures}\label{STMHD_sec_Fellerinv}

In this section we prove further properties of the tamed MHD equations. First, we show that they generate a Feller semigroup under stronger assumptions. Then, we prove that there exists an invariant measure for this Feller semigroup in the case of periodic boundary conditions.

We consider the \emph{time-homogeneous case}, i.e. the functions $f, \Sigma, H$ of our equations are assumed to be independent of time. Furthermore, we assume Lipschitz conditions on the first-order derivatives of the function $H$:
\begin{enumerate}[label=(H\arabic*)', ref=(H\arabic*)']
	\setcounter{enumi}{2}
	\item\label{STMHD_FP_itm_ass_H3} There exists a constant $C_{H} > 0$ and a function $F_{H} \in L^{1}(\mbD)$ such that for all $x \in \mbD$, $y,y' \in \R^{6}$, $j=1,2,3$ the following conditions hold:
	\begin{align*}
		\| \partial_{x^{j}} H(x,y) \|_{\ell^{2}}^{2} + \| H(x,y) \|_{\ell^{2}}^{2} &\leq C_{H} |y|^{2} + F_{H}(x), \\
		\| \partial_{x^{j}} H(x,y) - \partial_{x^{j}} H(x,y') \|_{\ell^{2}} &\leq C_{H} |y-y'|, \\
		\| \partial_{y^{j}} H(x,y) \|_{\ell^{2}} &\leq C_{H}, \\
		\| \partial_{y^{j}} H(x,y) - \partial_{(y')^{j}} H(x,y') \|_{\ell^{2}} &\leq C_{H} |y-y'|.
	\end{align*}
\end{enumerate}
For an initial condition $y_{0} \in \mcH^{1}$, let $y(t;y_{0})$ be the unique solution to \eqref{STMHD_eq_evol_eqn} with $y(0;y_{0}) = y_{0}$. Then by the uniqueness of solutions, we know that $\{ y(t;y_{0}) ~|~ y_{0} \in \mcH^{1}, t \geq 0 \}$ is a strong Markov process with state space $\mcH^{1}$. In proving the Feller property of the associated semigroup, we need the following result.
\begin{lemma}\label{STMHD_FP_thm_Lemma}
	For $y_{0}, y_{0}' \in \mcH^{1}$, $R > 0$, define the stopping times
	\begin{align*}
		\tau_{R}^{y_{0}} &:= \inf \{ t \geq 0 ~|~ \|y(t;y_{0})\|_{\mcH^{1}} > R \}, \\
		\tau_{R} := \tau_{R}^{y_{0}, y_{0}'} &:= \tau_{R}^{y_{0}} \wedge \tau_{R}^{y_{0}'}.
	\end{align*}
	Assume \ref{STMHD_prel_itm_ass_H1}, \ref{STMHD_prel_itm_ass_H2}, \ref{STMHD_FP_itm_ass_H3}. Then there is a constant $C_{t,R,N,f,H,\Sigma} > 0$ such that
	\begin{align*}
		\E \left[\| y(t \wedge \tau_{R};y_{0}) - y(t \wedge \tau_{R};y_{0}') \|_{\mcH^{1}}^{2} \right] \leq C_{t,R,N,f,H,\Sigma} \| y_{0} - y_{0}' \|_{\mcH^{1}}^{2}.
	\end{align*}
\end{lemma}
\begin{proof}
	We follow the proof of \cite[Lemma 4.1, p. 238 ff.]{RZ09b} For notational convenience, we denote $y(t) := y(t;y_{0})$, $\tilde{y}(t) := y(t;y_{0}')$, $z(t) := y(t) - \tilde{y}(t)$ and $t_{R} := t \wedge \tau_{R}$. By It\^{o}'s formula, we have
	\begin{align*}
		\| z(t_{R}) \|_{\mcH^{1}}^{2} &= \| z(0) \|_{\mcH^{1}}^{2} + 2 \int\limits_{0}^{t_{R}} \langle \mcA(y(s)) - \mcA(\tilde{y}(s)), z(s) \rangle_{\mcH^{1}} \diff s + 2 \int\limits_{0}^{t_{R}} \langle f(y(s)) - f(\tilde{y}(s)), z(s) \rangle_{\mcH^{1}} \diff s \\
		&\quad + 2 \sum_{k=1}^{\infty} \int\limits_{0}^{t_{R}} \langle \mcB_{k}(y(s)) - \mcB_{k}(\tilde{y}(s)), z(s) \rangle_{\mcH^{1}} \diff \mcW_{s}^{k}  + \sum_{k=1}^{\infty} \int\limits_{0}^{t_{R}} \| \mcB_{k}(y(s)) - \mcB_{k}(\tilde{y}(s)), z(s) \|_{\mcH^{1}}^{2} \diff s \\
		&=: \| z(0) \|_{\mcH^{1}}^{2} + I_{1}(t_{R}) + I_{2}(t_{R}) + I_{3}(t_{R}) + I_{4}(t_{R}).
	\end{align*}
	We denote by $\bm{z}_{v}$ the velocity component of $z$, i.e. $\bm{z}_{v} := \bm{v} - \tilde{\bm{v}}$, and similarly we write $\bm{z}_{B} := \bm{B} - \tilde{\bm{B}}$. Then $I_{1}(t_{R})$ has the following form:
	\begin{align*}
		I_{1}(t_{R}) &= 2 \int\limits_{0}^{t_{R}} \langle \D z, z \rangle_{\mcH^{1}} - 2 \int\limits_{0}^{t_{R}} \langle g_{N}(|y|^{2})y - g_{N}(|\tilde{y}|^{2})\tilde{y}, z \rangle_{\mcH^{1}} \\
		&\quad + 2 \int\limits_{0}^{t_{R}} \Big\{ \langle - (\bm{v} \cdot \nabla) \bm{v} + (\tilde{\bm{v}} \cdot \nabla) \tilde{\bm{v}}, \bm{z}_{v} \rangle_{\mcH^{1}} + \langle (\bm{B} \cdot \nabla) \bm{B} - (\tilde{\bm{B}} \cdot \nabla) \tilde{\bm{B}}, \bm{z}_{v} \rangle_{\mcH^{1}} \\
		&\quad\quad + \langle - (\bm{v} \cdot \nabla) \bm{B} + (\tilde{\bm{v}} \cdot \nabla) \tilde{\bm{B}}, \bm{z}_{B} \rangle_{\mcH^{1}} +  \langle (\bm{B} \cdot \nabla) \bm{v} - (\tilde{\bm{B}} \cdot \nabla) \tilde{\bm{v}}, \bm{z}_{B} \rangle_{\mcH^{1}} \Big\} \diff s.
	\end{align*}
	The first term is readily analysed:
	\begin{align*}
		2 \int\limits_{0}^{t_{R}} \langle \D z, z \rangle_{\mcH^{1}} = - 2 \int\limits_{0}^{t_{R}} \| z \|_{\mcH^{2}}^{2} \diff s + 2 \int\limits_{0}^{t_{R}} \| z \|_{\mcH^{1}}^{2} \diff s.
	\end{align*}
	For the second term, we find by using Young's inequality, $g_{N}(r) \leq Cr$, H\"older's inequality (with $p=3$, $q=3/2$) and the Sobolev embedding $\mcH^{1} \subset \mcL^{6}$ that for some $\theta \in \R$
	\begin{align*}
		&\langle g_{N}(|y|^{2})y - g_{N}(|\tilde{y}|^{2})\tilde{y}, z \rangle_{\mcH^{1}} \\
		&\leq \left\|  g_{N}(|y|^{2}) |z| \right\|_{L^{2}}\left\| z \right\|_{\mcH^{2}} + \left\| |g_{N}'(\theta)| \cdot |z| (|y| + |\tilde{y}|) |\tilde{y}| \right\|_{L^{2}} \left\| z \right\|_{\mcH^{2}} \\
		&\leq 2 \varepsilon \left\| z \right\|_{\mcH^{2}}^{2} + C_{\varepsilon,N}\left( \left\|  g_{N}(|y|^{2}) |z| \right\|_{L^{2}}^{2} + \left\| |z| (|y|^{2} + |\tilde{y}|^{2} ) \right\|_{L^{2}}^{2} \right) \\
		&\leq 2 \varepsilon \left\| z \right\|_{\mcH^{2}}^{2} + C_{\varepsilon,N} \| z \|_{\mcL^{6}}^{2} \left( \| y \|_{\mcL^{6}}^{4} + \| \tilde{y} \|_{\mcL^{6}}^{4} \right) \\
		&\leq 2 \varepsilon \left\| z \right\|_{\mcH^{2}}^{2} + C_{\varepsilon,N} \| z \|_{\mcH^{1}}^{2} \left( \| y \|_{\mcH^{1}}^{4} + \| \tilde{y} \|_{\mcH^{1}}^{4} \right) \\
		&\leq 2 \varepsilon \left\| z \right\|_{\mcH^{2}}^{2} + C_{\varepsilon,N,R} \| z \|_{\mcH^{1}}^{2}.
	\end{align*}
	The third term consists of four sub-terms, all of which are very similar. We thus only estimate one of them: Young's inequality, the Sobolev embedding $H^{1} \subset L^{6}$ as well as the Gagliardo-Nirenberg-Sobolev inequalities $\| \bm{u} \|_{L^{\infty}} \leq C \| \bm{u} \|_{H^{2}}^{3/4} \| \bm{u} \|_{L^{2}}^{1/4}$ and $\| \nabla \bm{u} \|_{L^{3}} \leq C \| \bm{u} \|_{H^{2}}^{3/4} \| \bm{u} \|_{L^{2}}^{1/4}$, combined with another application of Young's inequality (with $p=4/3$, $q=4$) yield
	\begin{align*}
		&\langle (\bm{v} \cdot \nabla) \bm{v} - (\tilde{\bm{v}} \cdot \nabla) \tilde{\bm{v}}, \bm{z}_{v} \rangle_{H^{1}} \leq 2 \varepsilon \| \bm{z}_{v} \|_{H^{2}}^{2} + C_{\varepsilon} \left( \| (\bm{z}_{v} \cdot \nabla) \bm{v} \|_{L^{2}}^{2} + \| (\tilde{\bm{v}} \cdot \nabla) \bm{z}_{v} \|_{L^{2}}^{2} \right) \\
		&\leq 2 \varepsilon \| \bm{z}_{v} \|_{H^{2}}^{2} + C_{\varepsilon} \left( \| \bm{z}_{v} \|_{L^{\infty}}^{2} \| \nabla \bm{v} \|_{L^{2}}^{2} + \| \tilde{\bm{v}} \|_{L^{6}}^{2} \| \nabla \bm{z}_{v} \|_{L^{3}}^{2} \right) \\
		&\leq 2 \varepsilon \| \bm{z}_{v} \|_{H^{2}}^{2} + C_{\varepsilon} \left( \| \bm{z}_{v} \|_{H^{2}}^{3/2} \| \bm{z}_{v} \|_{L^{2}}^{1/2} \| \bm{v} \|_{H^{1}}^{2} + \| \tilde{\bm{v}} \|_{H^{1}}^{2} \| \bm{z}_{v} \|_{H^{2}}^{3/2} \| \bm{z}_{v} \|_{L^{2}}^{1/2} \right) \\
		&\leq 4 \varepsilon \| \bm{z}_{v} \|_{H^{2}}^{2} + C_{\varepsilon,R}\| \bm{z}_{v} \|_{L^{2}}^{2} 		\leq 4 \varepsilon \| \bm{z}_{v} \|_{H^{2}}^{2} + C_{\varepsilon,R}\| \bm{z}_{v} \|_{H^{1}}^{2}.
	\end{align*}
	Analysing the other sub-terms in the same way and putting everything together, we find
	\begin{equation}\label{STMHD_FP_eq_lemma_I1}
		I_{1}(t_{R}) \leq - 2 \int\limits_{0}^{t_{R}} \| z(s) \|_{\mcH^{2}}^{2} \diff s + C_{\varepsilon,N,R} \int\limits_{0}^{t_{R}} \| z(s) \|_{\mcH^{1}}^{2} \diff s + 10 \varepsilon \int\limits_{0}^{t_{R}} \| z(s) \|_{\mcH^{2}}^{2} \diff s.
	\end{equation}
	$I_{2}(t_{R})$ is estimated using Young's inequality and Assumption \ref{STMHD_prel_itm_ass_H1}
	\begin{align*}
		I_{2}(t_{R}) 
		&\leq \varepsilon \int\limits_{0}^{t_{R}} \| z(s) \|_{\mcH^{2}}^{2} \diff s+ C_{\varepsilon}\int\limits_{0}^{t_{R}} \| f(y(s)) - f(\tilde{y}(s)) \|_{\mcH^{0}}^{2} \diff s \\
		&\leq \varepsilon \int\limits_{0}^{t_{R}} \| z(s) \|_{\mcH^{2}}^{2} \diff s+ C_{\varepsilon,f}\int\limits_{0}^{t_{R}} \| z(s) \|_{\mcH^{0}}^{2} \diff s.
	\end{align*}
	The term $I_{3}(t_{R})$ is a martingale and thus vanishes after taking expectations.
	
	For $I_{4}(t_{R})$, we use the following considerations, similar to the ones in the proof of Lemma \ref{STMHD_thm_estB}:
	\begin{align*}
		\| \mcB(y) - \mcB(\tilde{y}) \|_{L_{2}(\ell^{2}\times \ell^{2};\mcH^{1})}^{2} = \| \mcB(y) - \mcB(\tilde{y}) \|_{L_{2}(\ell^{2}\times \ell^{2};\mcH^{0})}^{2} + \left\| \nabla \left( \mcB(y) - \mcB(\tilde{y}) \right) \right\|_{L_{2}(\ell^{2}\times \ell^{2};\mcH^{0})}^{2},
	\end{align*}
	and the latter term consists (by using the chain rule) of terms of the following form:
	\begin{align*}
		&\partial_{x^{j}} \left( \mcB_{k}(y) - \mcB_{k}(\tilde{y}) \right) = \partial_{x^{j}} \left( \mcP \left( \Sigma_{k}(x) \cdot \nabla \right) z + \mcP \left( H_{k}(x,y) - H_{k}(x,\tilde{y}) \right) \right) \\
		&= \mcP \Big\{ \left( \left( \partial_{x^{j}} \Sigma_{k}(x) \right) \cdot \nabla \right) z + \left( \Sigma_{k}(x) \cdot \nabla \right) \partial_{x^{j}} z 
		+  (\partial_{x^{j}} H_{k})(x,y) - (\partial_{x^{j}} H_{k})(x,\tilde{y}) \\
		&\quad + \sum_{i=1}^{6} \left[ \left( \partial_{y^{i}} H_{k}(x,y) \right)\partial_{x^{j}} z^{i} - \left\{ \left( \partial_{y^{i}} H_{k}(x,y) \right) - \left( \partial_{\tilde{y}^{i}} H_{k}(x,\tilde{y}) \right) \right\}\partial_{x^{j}} \tilde{y}^{i} \right]	\Big\}.
	\end{align*}
	Thus we find, using Assumptions \ref{STMHD_prel_itm_ass_H2} and \ref{STMHD_FP_itm_ass_H3}, as well as Equation \eqref{STMHD_prel_eq_Besselpot_est}, the Gagliardo-Nirenberg inequality and Young's inequality, that
	\begin{align*}
		&\| \mcB(y) - \mcB(\tilde{y}) \|_{L_{2}(\ell^{2}\times \ell^{2};\mcH^{1})}^{2} \\
		&\leq 2 \sum_{k=1}^{\infty} \left\| \left( \Sigma_{k} \cdot \nabla \right) z \right\|_{\mcH^{0}}^{2} + \| H_{k}(y) - H_{k}(\tilde{y}) \|_{\mcH^{0}}^{2} \\
		&\quad + \sum_{k=1}^{\infty} \int_{\mbD} \sum_{j=1}^{3} \Big|  \left( \left( \partial_{x^{j}} \Sigma_{k}(x) \right) \cdot \nabla \right) z + \left( \Sigma_{k}(x) \cdot \nabla \right) \partial_{x^{j}} z 
		+  (\partial_{x^{j}} H_{k})(x,y) - (\partial_{x^{j}} H_{k})(x,\tilde{y}) \\
		&\quad + \sum_{i=1}^{6} \left[ \left( \partial_{y^{i}} H_{k}(x,y) \right)\partial_{x^{j}} z^{i} - \left\{ \left( \partial_{y^{i}} H_{k}(x,y) \right) - \left( \partial_{\tilde{y}^{i}} H_{k}(x,\tilde{y}) \right) \right\}\partial_{x^{j}} \tilde{y}^{i} \right]  \Big|^{2} \diff x \\
		&\leq 2 \sup_{x \in \mbD} \left\|\Sigma_{k}(x) \right\|_{\ell^{2}}^{2} \| \nabla z \|_{\mcH^{0}}^{2} + C_{H} \| z \|_{\mcH^{0}}^{2} + 2 \sum_{k=1}^{\infty} \int_{\mbD} \sum_{j=1}^{3} \Big|  \left( \Sigma_{k}(x) \cdot \nabla \right) \partial_{x^{j}} z \Big|^{2} \diff x \\
		&\quad + 2\sum_{k=1}^{\infty} \int_{\mbD} \sum_{j=1}^{3} \Big|  \left( \left( \partial_{x^{j}} \Sigma_{k}(x) \right) \cdot \nabla \right) z 
		+  (\partial_{x^{j}} H_{k})(x,y) - (\partial_{x^{j}} H_{k})(x,\tilde{y}) \\
		&\quad + \sum_{i=1}^{6} \left( \partial_{y^{i}} H_{k}(x,y) \right)\partial_{x^{j}} z^{i} - \sum_{i=1}^{6} \left\{ \left( \partial_{y^{i}} H_{k}(x,y) \right) - \left( \partial_{\tilde{y}^{i}} H_{k}(x,\tilde{y}) \right) \right\}\partial_{x^{j}} \tilde{y}^{i}  \Big|^{2} \diff x \\
		&\leq 2 \sup_{x \in \mbD} \left\|\Sigma_{k}(x) \right\|_{\ell^{2}}^{2} \| z \|_{\mcH^{1}}^{2} + C_{H} \| z \|_{\mcH^{0}}^{2} + 2 \sup_{x \in \mbD} \left\|\Sigma_{k}(x) \right\|_{\ell^{2}}^{2} \sum_{j,l=1}^{3} \|  \partial_{x^{l}} \partial_{x^{j}} z \|_{\mcL^{2}}^{2} \\
		&\quad + 8 \sup_{x \in \mbD} \| \nabla \Sigma_{k}(x) \|_{\ell^{2}}^{2} \| \nabla z \|_{\mcH^{0}}^{2}
		+  C_{H} \| z \|_{\mcH^{1}}^{2} +  C_{H}  \int_{\mbD} \sum_{j=1}^{3} \sum_{i=1}^{6} |z|^{2} | \partial_{x^{j}} \tilde{y}^{i} |^{2} \diff x \\
		&\leq 2d^{2} \sup_{x \in \mbD} \left\|\Sigma_{k}(x) \right\|_{\ell^{2}}^{2}  \| z \|_{\mcH^{2}}^{2}	+  C_{\Sigma,H} \| z \|_{\mcH^{1}}^{2} + C_{H} \| z \|_{\mcL^{\infty}}^{2} \| \tilde{y} \|_{\mcH^{1}}^{2} \\
		&\leq (\frac{1}{2} + \varepsilon) \| z \|_{\mcH^{2}}^{2}	+ C_{\Sigma,H} \| z \|_{\mcH^{1}}^{2} + C_{H} \| z \|_{\mcH^{0}}^{2} \| \tilde{y} \|_{\mcH^{1}}^{8}.
	\end{align*}
	Integrating over time we finally get
	\begin{align*}
		I_{4}(t_{R}) \leq \int\limits_{0}^{t_{R}} (\frac{1}{2} + \varepsilon) \| z \|_{\mcH^{2}}^{2}	+  C_{H,R} \| z \|_{\mcH^{0}}^{2} +  C_{\Sigma,H} \| z \|_{\mcH^{1}}^{2} \diff s.
	\end{align*}
	Thus, adding all the contributions together,
	\begin{align*}
		&\E \left[\| z(t_{R}) \|_{\mcH^{1}}^{2} \right] \\ 
		&\leq \| z(0) \|_{\mcH^{1}}^{2} - (3/2 - 12\varepsilon) \E \left[ \int\limits_{0}^{t_{R}} \| z(s) \|_{\mcH^{2}}^{2} \diff s \right] + C_{\varepsilon,R,N,f,H,\Sigma} \E \left[ \int\limits_{0}^{t_{R}} \| z(s) \|_{\mcH^{1}}^{2} \diff s \right].
	\end{align*}
	Choosing $\varepsilon = \frac{1}{8}$, we find
	\begin{align*}
		\E \left[\| z(t \wedge \tau_{R}) \|_{\mcH^{1}}^{2} \right] \leq \| z(0) \|_{\mcH^{1}}^{2} + C_{R,N,f,H,\Sigma}  \int\limits_{0}^{t} \E \left[ \| z(s \wedge \tau_{R}) \|_{\mcH^{1}}^{2} \diff s \right].
	\end{align*}
	An application of Gronwall's Lemma then yields the desired result.
\end{proof}

Let $BC_{\text{loc}}(\mcH^{1})$ denote the set of bounded, locally uniformly continuous functions on $\mcH^{1}$. The supremum norm
\begin{align*}
	\| \phi \|_{\infty} := \sup_{y \in \mcH^{1}} | \phi(y) |
\end{align*}
turns this space into a Banach space.

For $t \geq 0$, define the semigroup $T_{t}$ associated with the Markov process $\{ y(t;y_{0}) ~|~ y_{0} \in \mcH^{1}, t \geq 0 \}$ by
\begin{align*}
	T_{t} \phi(y_{0}) := \E \left[ \phi(y(t;y_{0})) \right], \quad \phi \in BC_{\text{loc}}(\mcH^{1}).
\end{align*}
Using the previous lemma, we show that this is a Feller semigroup.
\begin{theorem}[Feller property]\label{STMHD_thm_Feller}
	Under the Assumptions \ref{STMHD_prel_itm_ass_H1}, \ref{STMHD_prel_itm_ass_H2} and \ref{STMHD_FP_itm_ass_H3}, for every $t \geq 0$, $T_{t}$ maps $BC_{\text{loc}}(\mcH^{1})$ into itself, i.e. it is a Feller semigroup on $BC_{\text{loc}}(\mcH^{1})$.
\end{theorem}
\begin{proof}
    Having formulated the problem in this abstract way, we can use Lemma \ref{STMHD_FP_thm_Lemma}  to run the exact same proof as in \cite[Theorem 4.2, p. 241 f.]{RZ09b} to prove the claim, as it does not use any specific properties of the Navier--Stokes equations.
\end{proof}

In the periodic case, we can show existence of an invariant measure for our equations:
\begin{theorem}[Invariant measures in the periodic case]\label{STMHD_thm_inv_ms_ex}
	Under the hypotheses \ref{STMHD_prel_itm_ass_H1}, \ref{STMHD_prel_itm_ass_H2}, \ref{STMHD_FP_itm_ass_H3}, in the periodic case $\mbD = \mbT^{3}$, there exists an invariant measure $\mu \in \mcP(\mcH^{1})$ associated to $(T_{t})_{t \geq 0}$ such that for every $t \geq 0$, $\phi \in BC_{\text{loc}}(\mcH^{1})$
	\begin{align*}
		\int_{\mcH^{1}} T_{t} \phi(y_{0}) \diff \mu(y_{0}) = \int_{\mcH^{1}} \phi(y_{0}) \diff \mu(y_{0}).
	\end{align*}
\end{theorem}
\begin{proof}
    We start by using It\^{o}'s formula for $\| y(t) \|_{\mcH^{0}}^{2}$, Equation \eqref{STMHD_prel_eq_A_testing2}, Young's inequality, Equation \eqref{STMHD_prel_eq_B_L2l2H0} as well as Assumption \ref{STMHD_prel_itm_ass_H1} and -- owing to the boundedness of the domain $\mbT^{3}$ -- the embedding $\mcL^{4}(\mbT^{3}) \subset \mcL^{2}(\mbT^{3})$. These combined yield that
	\begin{equation}\label{STMHD_FP_eq_aprioriH0}
		\begin{split}
		\E \left[ \| y(t) \|_{\mcH^{0}}^{2} \right] + \int\limits_{0}^{t} \E \left[ \| y(s) \|_{\mcH^{1}}^{2} \right] \diff s + \int\limits_{0}^{t} \E \left[ \| y(s) \|_{\mcL^{4}}^{4} \right] \diff s \leq C_{f,H,N,\mbD} \left( \| y(0) \|_{\mcH^{0}}^{2} +  t \right).
		\end{split}
	\end{equation}
	Now using the It\^{o} formula  for the norm $\| y(t) \|_{\mcH^{1}}^{2}$ and the corresponding higher-order estimates \eqref{STMHD_prel_eq_A_H1_testing}, \eqref{STMHD_prel_eq_B_L2l2H1}, we find by proceeding in the same way as before -- additionally using \eqref{STMHD_FP_eq_aprioriH0} -- that
	\begin{align*}
		\E \left[ \| y(t) \|_{\mcH^{1}}^{2} \right] \leq - \frac{1}{2} \int\limits_{0}^{t} \E \left[ \| y(s) \|_{\mcH^{2}}^{2} \right] \diff s + C_{f,H,N,\mbD} \cdot t,
	\end{align*}
	which implies
	\begin{equation}\label{STMHD_eq_invmstightness}
		\frac{1}{t} \int\limits_{0}^{t} \E \left[ \| y(s) \|_{\mcH^{2}}^{2} \right] \diff s \leq C_{f,H,N,\mbD}.
	\end{equation}
	Note that this inequality is uniform in $t$ and in the initial condition $y_{0}$.
	
We now conclude as follows: following the notation of G. Da Prato and J. Zabczyk \cite{DPZ96}, we define the kernels
	\begin{align*}
		R_{T}(y_{0},\Gamma) := \frac{1}{T} \int\limits_{0}^{T} (T_{s}1_{\Gamma})(y_{0}) \diff t, \quad y_{0} \in \mcH^{1}, \Gamma \in \mcB(\mcH^{1}).
	\end{align*}
Set $E = \mcH^{1}$ and let $\nu \in \mcP(\mcH^{1})$ be an arbitrary probability measure. Define the associated measure
\begin{align*}
	R_{T}^{*} \nu (\Gamma) = \int_{E} R_{T}(y_{0},\Gamma) \nu(\diff y_{0}).
\end{align*}
We will show that these measures are tight. To this end, consider the sets $\Gamma_{r} := \left\{ y_{0} \in \mcH^{2} ~|~ \| y_{0} \|_{\mcH^{2}} \leq r \right\}$ $\subset \mcH^{2} \subset \mcH^{1}$. Since they are bounded subsets of $\mcH^{2}$ and the embedding $\mcH^{2} \subset \mcH^{1}$ is compact, they are compact subsets of $\mcH^{1}$. Now, using 	\eqref{STMHD_eq_invmstightness}, we show that the measures $R_{T}^{*}\nu$ are concentrated on $\Gamma_{r}$ for sufficiently large $r > 0$, hence tight: denoting the complement of a set $A \subset E$ by $A^{c}$ and using Chebychev's inequality, we find 
\begin{align*}
	R_{T}^{*} \nu(\Gamma_{r}^{c}) &= \int\limits_{E} R_{T}(y_{0},\Gamma_{r}^{c}) \nu(\diff y_{0}) = \int_{E} \frac{1}{T} \int\limits_{0}^{T} \E \left[ 1_{\Gamma_{r}^{c}}(y(t;y_{0})) \right] \diff t~ \nu(\diff y_{0}) \\
	&\leq \int_{E} \frac{1}{T} \int\limits_{0}^{T} \E \left[ \frac{\| y(t;y_{0}) \|_{\mcH^{2}}^{2}}{r^{2}} \right] \diff t~ \nu(\diff y_{0}) \leq \frac{1}{r^{2}}\int_{E} C_{f,H,N,\mbD} ~\nu(\diff y_{0}) = \frac{C_{f,H,N,\mbD}}{r^{2}}.
\end{align*}
Therefore, the Krylov-Bogoliubov theorem \cite[Corollary 3.1.2, p. 22]{DPZ96} ensures the existence of an invariant measure.
\end{proof}
\begin{remark}
	Note that Z. Brze\'{z}niak and G. Dhariwal  \cite{BD19} prove existence of an invariant measure for a similar system even in the case of the full space, so there might be hope to extend this theorem also in our case.
\end{remark}

\section*{Acknowledgements}
Financial support by the German Research Foundation (DFG) through the IRTG 2235 is gratefully acknowledged. The author would further like to thank Michael R\"ockner for helpful discussions.

 \bibliographystyle{plain}
\bibliography{refs}
 \end{document}